\documentclass[11pt]{amsart}

\usepackage{epsfig,graphics,color}
\usepackage{amssymb, amscd}
\usepackage[backref=page]{hyperref}
\usepackage[matrix,arrow,curve]{xy}

\newtheorem{theorem}{Theorem}[section]

\newtheorem{cor}[theorem]{Corollary}
\newtheorem{lma}[theorem]{Lemma}
\newtheorem*{question}{Question}

\newtheorem{corollary}[theorem]{Corollary}
\newtheorem{lemma}[theorem]{Lemma}

\theoremstyle{definition}

\theoremstyle{remark}
\newtheorem{remark}[theorem]{Remark}
\newtheorem{rmk}[theorem]{Remark}

\numberwithin{equation}{section}

\newcommand{\R}{{\mathbb{R}}}
\newcommand{\C}{{\mathbb{C}}}

\newcommand{\Z}{{\mathbb{Z}}}

\newcommand{\bo}{\boldsymbol}

\newcommand{\CZ}{\operatorname{CZ}}

\newcommand{\Hom}{\operatorname{Hom}}

\newcommand{\ev}{\operatorname{ev}}

\newcommand{\pa}{\partial}
\newcommand{\p}{\partial}

\newcommand{\ind}{\operatorname{index}}

\newcommand{\coker}{\operatorname{coker}}

\begin{document}

\title{Symplectic and contact differential graded algebras}

\author{Tobias Ekholm}
\address{Uppsala University, Box 480, 751 06 Uppsala, Sweden, and Institute Mittag-Leffler, Aurav 17, 182 60 Djursholm, Sweden}
\email{tobias.ekholm@math.uu.se}
\thanks{T.E. partially funded by the Knut and Alice Wallenberg Foundation as a
Wallenberg scholar and by the Swedish Research Council, 2012-2365}
\author{Alexandru Oancea}
\address{Sorbonne Universit\'es, UPMC Univ Paris 06, UMR 7586, Institut de 
Math\'e\-ma\-tiques de Jussieu-Paris Rive Gauche, Case 247, 4 place Jussieu, F-75005, Paris, France}
\email{alexandru.oancea@imj-prg.fr}
\thanks{A.O. partially funded by the European Research Council, StG-259118-STEIN}

\date{June 2, 2015}

\begin{abstract}
We define Hamiltonian simplex differential graded algebras (DGA) with differentials that deform the high energy symplectic homology differential and wrapped Floer homology differential in the cases of closed and open strings in a Weinstein manifold, respectively. The order $m$ term in the differential is induced by varying natural degree $m$ co-products over an $(m-1)$-simplex, where the operations near the boundary of the simplex are trivial. We show that the Hamiltonian simplex DGA is quasi-isomorphic to the (non-equivariant) contact homology algebra and to the Legendrian homology algebra of the ideal boundary in the closed and open string cases, respectively.    
\end{abstract}

\maketitle


\section{Introduction}\label{S:intr}
Let $X$ be a Weinstein manifold and let $L\subset X$ be an exact Lagrangian submanifold. (We use the terminology of~\cite{Cieliebak-Eliashberg} for Weinstein manifolds, cobordisms etc.~throughout the paper.) Assume that $(X,L)$ is cylindrical at infinity meaning that outside a compact set $(X,L)$ looks like $([0,\infty)\times Y,[0,\infty)\times\Lambda)$, where $Y$ is a contact manifold, $\Lambda\subset Y$ a Legendrian submanifold, and the Liouville form on $[0,\infty)\times Y$ is the symplectization form $e^{t}\alpha$, for $\alpha$ a contact form on $Y$ and $t$ the standard coordinate in $[0,\infty)$.

There are a number of Floer homological theories associated to this geometric situation. For example there is \emph{symplectic homology} $SH(X)$ which can be defined~\cite{Viterbo99,Seidel_biased,BOauto} using a time-dependent Hamiltonian $H\colon X\times I\to\R$, $I=[0,1]$ which is a small perturbation of a time independent Hamiltonian that equals a small positive constant in the compact part of $X$ and is linearly increasing of certain slope in the coordinate $r=e^{t}$ in the cylindrical end at infinity, and then taking a certain limit over increasing slopes. The chain complex underlying $SH(X)$ is denoted $SC(X)$ and is generated by the $1$-periodic orbits of the Hamiltonian vector field $X_{H}$ of $H$, graded by their Conley-Zehnder indices. These fall into two classes: low energy orbits in the compact part of $X$ and (reparameterizations of) Reeb orbits of $\alpha$ in the region in the end where $H$ increases from a function that is close to zero to a function of linear growth. The differential counts Floer holomorphic cylinders interpolating between the orbits. These are solutions $u\colon \R\times S^{1}\to X$, $S^{1}=I/\partial I$ of the Floer equation
\begin{equation}\label{eq:Floer}
(du-X_{H}\otimes dt)^{0,1}=0,
\end{equation}  
where $s+it\in \R\times S^{1}$ is a standard complex coordinate and the complex anti-linear part is taken with respect to a chosen adapted almost complex structure $J$ on $X$. The $1$-periodic orbits of $H$ are closed loops that are critical points of an action functional, and cylinders solving \eqref{eq:Floer} are similar to instantons that capture the effect of tunneling between critical points. Because of this and analogies with (topological) string theory, we say that symplectic homology is a \emph{theory of closed strings}.

The open string analogue of $SH(X)$ is a corresponding theory for paths with endpoints in the Lagrangian submanifold $L\subset X$. It is called the \emph{wrapped Floer homology} of $L$ and here denoted $SH(L)$. Its underlying chain complex $SC(L)$ is generated by Hamiltonian time $1$ chords that begin and end on $L$, graded by a Maslov index. Again these fall into two classes: high energy chords that correspond to Reeb chords of the ideal Legendrian boundary $\Lambda$ of $L$ and low energy chords that correspond to critical points of $H$ restricted to $L$. The differential on $SC(L)$ counts Floer holomorphic strips with boundary on $L$ interpolating between Hamiltonian chords, i.e.~solutions 
\[  
u\colon\left(\R\times I,\partial(\R\times I)\right)\to (X,L) 
\]
of \eqref{eq:Floer}.

We will also consider a mixed version of open and closed strings. The graded vector space underlying the chain complex is simply $SC(X,L)=SC(X)\oplus SC(L)$, and the  differential $d_{1}\colon SC(X,L)\to SC(X,L)$ has the following matrix form with respect to this decomposition (subscripts ``$\mathrm{c}$'' and ``$\mathrm{o}$'' refer to \emph{closed} and \emph{open}, respectively):
\[
d_1=\left(\begin{matrix}d_{\mathrm{cc}} & d_{\mathrm{oc}}\\ 0 & d_{\mathrm{oo}} \end{matrix}\right).
\]
Here $d_{\mathrm{cc}}$ and $d_{\mathrm{oo}}$ are the differentials on $SC(X)$ and $SC(L)$, respectively, and $d_{\mathrm{oc}}\colon SC(L)\to SC(X)$ is a chain map of degree $-1$.  Each of these three maps counts solutions of \eqref{eq:Floer} on a Riemann surface with two punctures, one positive regarded as input, and one negative regarded as output. For $d_{\mathrm{cc}}$ the underlying Riemann surface is the cylinder, for $d_{\mathrm{oo}}$ the underlying Riemann surface is the strip, and for $d_{\mathrm{oc}}$ the underlying Riemann surface is the cylinder $\R\times S^{1}$ with a slit at $[0,\infty)\times \{1\}$ (or equivalently, a disk with two boundary punctures, a sphere with two interior punctures, and a disk with positive boundary puncture and negative interior puncture).  We will denote the corresponding homology $SH(X,L)$. 

In order to count the curves in the differential over integers we use index bundles to orient solution spaces and for that we assume that the pair $(X,L)$ is relatively spin, see \cite{FO3}. 
As the differential counts Floer-holomorphic curves, it respects the energy filtration and the subspace generated by the low energy chords and orbits is a subcomplex. We denote the corresponding high energy quotient $SC^{+}(X,L)$ and its homology $SH^{+}(X,L)$. We define similarly $SC^{+}(X)$, $SC^{+}(L)$, $SH^{+}(X)$, and $SH^{+}(L)$.

In the context of Floer homology, the cylinders and strips above are the most basic Riemann surfaces, and it is well-known that more complicated Riemann surfaces $\Sigma$ can be included in the theory as follows, see~\cite{Seidel_biased, Ritter}. Pick a family of 1-forms $B$ with values in Hamiltonian vector fields on $X$ over the appropriate Deligne-Mumford space of domains and count rigid solutions of the Floer equation
\begin{equation}\label{eq:floereqintr}
(du-B)^{0,1}=0,
\end{equation} 
where $B(s+it)= X_{H_t}\otimes dt$ in cylindrical coordinates $s+it$ near the punctures of $\Sigma$.
The resulting operation descends to homology as a consequence of gluing and Gromov-Floer compactness. A key condition for solutions of \eqref{eq:floereqintr} to have relevant compactness properties is that $B$ is required to be non-positive in the following sense.
For each $x\in X$ we get a 1-form on $\Sigma$ with values in $T_{x}X$, $B(x)=X_{H_{t}^{z}}(x)\otimes\beta$, where $H_{t}^{z}\colon X\to\R$, $t\in I$ is a family of time dependent Hamiltonian functions parameterized by $z\in\Sigma$ and $\beta$ is a 1-form on $\Sigma$. The non-positivity condition is then that the 2-form associated to $B$, $d(H_{t}^{z}(x)\,\beta)$, is a non-positive multiple of the area form on $\Sigma$ for each $x\in X$. 

The most important such operations on $SH(X)$ are the BV-operator and the pair-of-pants product. The BV-operator corresponds to solutions of a parameterized Floer equation analogous to \eqref{eq:Floer} which twists the cylinder one full turn. The pair-of-pants product corresponds to a sphere with two positive and one negative puncture and restricts to the cup product on the ordinary cohomology of $X$, which here appears as the low energy part of $SH(X)$. Analogously, on $SH(L)$ the product corresponding to the disk with two positive and one negative boundary puncture restricts to the cup product on the cohomology of $L$, and the disk with one positive interior puncture and two boundary punctures of opposite signs expresses $SH(L)$ as a module over $SH(X)$. 

Seidel~\cite{Seidel_biased} showed that such operations are often trivial on $SH^{+}(X)$. Basic examples of this phenomenon are the operations $D_m$ given by disks and spheres with one positive and $m\ge 2$ negative punctures. By pinching the $1$-form $B$ in \eqref{eq:floereqintr} in the cylindrical end at one of the $m$ negative punctures, it follows that up to homotopy $D_m$ factors through the low energy part of the complex $SC(X,L)$. In particular, on the high energy quotient $SC^{+}(X,L)$ the operation is trivial if the $1$-form is pinched near at least one negative puncture. 

The starting point for this paper is to study operations $d_m$ that are associated to natural families of forms $B$ that interpolate between all ways of pinching near negative punctures. More precisely, for disks and spheres with one positive and $m$ negative punctures, we take $B$ in \eqref{eq:floereqintr} to have the form $B=X_{H}\otimes w_j\,dt$ in the cylindrical end, with coordinate $s+it$ in $[0,\infty)\times I$ for open strings and in $[0,\infty)\times S^{1}$ for closed strings, near the $j^{\rm th}$ puncture. Here $w_j$ is a positive function with a minimal value called \emph{weight}. By Stokes' theorem, in order for $B$ to satisfy the non-positivity condition, the sum of weights at the negative ends must be greater than the weight at the positive end. Thus the choice of $1$-form is effectively parameterized by an $(m-1)$-simplex and the equation~\eqref{eq:floereqintr} associated to a form which lies in a small neighborhood of the boundary of the simplex, where at least one weight is very small, has no solutions with all negative punctures at high energy chords or orbits. The operation $d_m$ is then defined by counting rigid solutions of \eqref{eq:floereqintr} where $B$ varies over the simplex bundle. Equivalently, we count solutions with only high energy asymptotes in the class dual to the fundamental class of the sphere bundle over Deligne-Mumford space obtained as the quotient space after fiberwise identification of the boundary of the simplex to a point. In particular, curves contributing to $d_m$ have formal dimension $-(m-1)$.  

Our first result says that the operations $d_m$ combine to give a DGA differential. The \emph{Hamiltonian simplex DGA}  $\mathcal{SC}^{+}(X,L)$ is the unital algebra generated by the generators of $SC^{+}(X,L)$ with grading shifted down by $1$, where orbits sign commute with orbits and chords but where chords do not commute. Let $d\colon \mathcal{SC}^{+}(X,L)\to\mathcal{SC}^{+}(X,L)$ be the map defined on generators $b$ by
\[ 
d\, b=d_1b+d_2b+\dots+d_mb+\cdots,
\]            
and extend it by the Leibniz rule.

\begin{theorem}\label{t:2ndary}
The map $d$ is a differential, $d\circ d=0$, and the homotopy type of the Hamiltonian simplex DGA $\mathcal{SC}^{+}(X,L)$ depends only on $(X,L)$. Furthermore, $\mathcal{SC}^{+}(X,L)$ is functorial in the following sense. 
If $(X_0,L_0)=(X,L)$, if $(X_{10},L_{10})$ is a Weinstein cobordism with negative end $(\partial X_0,\partial L_0)$, and if $(X_1,L_1)$ denotes the Weinstein manifold obtained by gluing $(X_{10},L_{10})$ to $(X_{0},L_{0})$, then there is a DGA map
\[
\Phi_{X_{10}}\colon \mathcal{SC}^{+}(X_1,L_1)\to\mathcal{SC}^{+}(X_0,L_0),
\]
and the homotopy class of this map is an invariant of $(X_{10},L_{10})$ up to Weinstein homotopy.  
\end{theorem}
If $L=\varnothing$ in Theorem \ref{t:2ndary} then we get a Hamiltonian simplex DGA $\mathcal{SC}^+(X)$ generated by high energy Hamiltonian orbits. This DGA is (graded) commutative. Also, the quotient $\mathcal{SC}^+(L)$ of $\mathcal{SC}^+(X,L)$ by the ideal generated by orbits is a Hamiltonian simplex DGA generated by high energy chords of $L$. We write $\mathcal{SH}^+(X,L)$ for the homology DGA of $\mathcal{SC}^+(X,L)$, and use the notation $\mathcal{SH}^{+}(X)$ and $\mathcal{SH}^+(L)$ with a similar meaning. If $X$ is the cotangent bundle of a manifold $X=T^{\ast}M$ then $SH(X)$ is isomorphic to the homology of the free loop space of $M$, see \cite{Viterbo-loop,SW,AS,Abouzaid-cotangent}, and the counterpart of $d_2$ in string topology is non-trivial, see~\cite{Goresky_Hingston}. 

Our second result expresses $\mathcal{SC}^{+}(X,L)$ in terms of the ideal boundary $(Y,\Lambda)=(\partial X,\partial L)$. Recall that the usual contact homology DGA $\widetilde{\mathcal{A}}(Y,\Lambda)$ is generated by closed Reeb orbits in $Y$ and by Reeb chords with endpoints on $\Lambda$, see \cite{EGH}. Here we use the differential that is naturally augmented by rigid once-punctured spheres in $X$ and by rigid once-boundary punctured disks in $X$ with boundary in $L$. (In the terminology of \cite{BEE} the differential counts anchored spheres and disks). 
In~\cite{Bourgeois_Oancea_sequence} a non-equivariant version of linearized orbit contact homology was introduced. In Section~\ref{sec:CHSH} we extend this construction and define a non-equivariant DGA that we call $\mathcal{A}(Y,\Lambda)$, which is generated by decorated Reeb orbits and by Reeb chords. We give two definitions of the differential on $\mathcal{A}(Y,\Lambda)$, one using Morse-Bott curves and one using curves holomorphic with respect to a domain dependent almost complex structure. In analogy with the algebras considered above we write $\mathcal{A}(Y)$ for the subalgebra generated by decorated orbits and $\mathcal{A}(\Lambda)$ for the quotient by the ideal generated by decorated orbits. 
 
In Sections \ref{sec:hamilt1form} and \ref{sec:nonequivDGA} we introduce a continuous $1$-parameter deformation of the simplex family of $1$-forms $B$ that turns off the Hamiltonian term in \eqref{eq:floereqintr} by sliding its support to the negative end in the domains of the curves and that leads to the following result. 

\begin{theorem} \label{thm:q-iso}
The deformation that turns the Hamiltonian term off gives rise to a DGA map
\[
\Phi\colon \mathcal{A}(Y,\Lambda)\to\mathcal{SC}^{+}(X,L).
\]
The map $\Phi$ is a quasi-isomorphism that takes the orbit subalgebra $\mathcal{A}(Y)$ quasi-isomorphically to the orbit subalgebra $\mathcal{SC}^{+}(X)$. Furthermore, it descends to the quotient $\mathcal{A}(\Lambda)$ and maps it to $\mathcal{SC}^{+}(L)$ as a quasi-isomorphism.
\end{theorem}   

The usual (equivariant) contact homology DGA $\widetilde{\mathcal{A}}(Y,\Lambda)$ is also quasi-isomorphic to a Hamiltonian simplex DGA that corresponds to a version of symplectic homology defined by a time independent Hamiltonian, see Theorem \ref{t:isoequiv}. For the corresponding result on the linear level, see~\cite{BOcontact-equivariant}.

\begin{rmk} \label{rmk:perturbations}
As is well-known, the constructions of the DGAs $\widetilde{\mathcal{A}}(Y,\Lambda)$ and ${\mathcal{A}}(Y,\Lambda)$, of the orbit augmentation induced by $X$, and of symplectic homology for time independent Hamiltonians require the use of abstract perturbations for the pseudo-holomorphic curve equation in a manifold with cylindrical end. This is an area where much current research is being done and there are several approaches, some of an analytical character, see e.g.~\cite{HWZ,HWZ-book1}, others of more algebraic topological flavor, see e.g.~\cite{Pardon}, and others of more geometric flavor, see e.g.~\cite{FO3}. Here we will not enter into the details of this problem but merely assume such a perturbation scheme has been fixed. Our results are independent of the nature of the perturbation scheme and use only the weakest form of it that allows us to count rigid curves over the rationals. In fact, in this spirit, Theorem~\ref{thm:q-iso} can be interpreted as an alternative definition of the (non-equivariant) contact homology DGA $\mathcal{A}(Y,\Lambda)$ that does not involve abstract perturbations. 
\end{rmk}

Theorem \ref{thm:q-iso} relates Symplectic Field Theory (SFT) and Hamiltonian Floer Theory. On the linear level the relation is rather direct, see \cite{Bourgeois_Oancea_sequence}, but not for the SFT DGA. The first candidate for a counterpart on the Hamiltonian Floer side collects the standard co-products to a DGA-differential, but that DGA is trivial by pinching. To see that recall the sphere bundle over Deligne-Mumford space obtained by identifying the boundary points in each fiber of the simplex bundle. The co-product DGA then corresponds to counting curves lying over the homology class of a point in each fiber, but that point can be chosen as the base point where all operations are trivial. The object that is actually isomorphic to the SFT DGA is the Hamiltonian simplex DGA related to the fundamental class of the spherization of the simplex bundle. 

In light of this, the following picture of the relation between Hamiltonian Floer Theory and SFT emerges. The Hamiltonian Floer Theory holomorphic curves solve a Cauchy-Riemann equation with Hamiltonian $0$-order term chosen consistently over Deligne-Mumford space. These curves are less symmetric than their counterparts in SFT, which are defined without additional $0$-order term. Accordingly, the moduli spaces of Hamiltonian Floer Theory have more structure and carry natural actions, e.g.~of scaling simplices and the framed little disk operad, see Section~\ref{sec:examples}. The SFT moduli spaces are in a sense homotopic to certain essential strata inside the Hamiltonian Floer Theory moduli spaces, and the structure and operations that they carry are intimately related to the natural actions mentioned. From this perspective, this paper studies the action given by scaling simplices in the most basic case of higher co-products. 

We end the introduction by a comparison between our constructions and other well-known constructions in Floer theory. In the case of open strings, the differential $d=\sum_{j=1}^{\infty}d_{j}$ can be thought of as a sequence of operations $(d_1,d_2,\dots,d_m,\dots)$ on the vector space $SC^{+}(L)$. These operations define the structure of an $\infty$-coalgebra on $SC^{+}(L)$ (with grading shifted down by one) and $\mathcal{SC}(L)^+$ is the cobar construction for this $\infty$-coalgebra. This point of view is dual to that of the Fukaya category, in which the primary objects of interest are $\infty$-algebras. In the Fukaya category setting algebraic invariants are obtained by applying (variants of) the Hochschild homology functor. In the DGA setting invariants are obtained more directly, as the homology of the Hamiltonian simplex DGA.

\subsubsection*{Acknowledgements} Both authors would like to thank the organizers of the G\"okova 20th Geometry and Topology Conference held in May 2013 for an inspiring meeting, during which the first ideas related to this paper crystallized. An early January 2013 discussion between the second author and Mohammed Abouzaid about the symplectic homology coproduct was also important. Part of this work was carried out while A.O. visited the Simons Center for Geometry and Physics at Stony Brook in the summer 2014.

\section{Simplex bundles over Deligne-Mumford space, splitting compatibility, and $1$-forms}
\label{sec:simplex_bundles} 
The Floer theories we study use holomorphic maps of disks and spheres with one positive and several negative punctures. Configuration spaces for such maps naturally fiber over the corresponding Deligne-Mumford space that parameterizes their domains. In this section we endow the Deligne-Mumford space with additional structure needed to define the relevant solution spaces. More precisely, we parameterize $1$-forms with non-positive exterior derivative by a simplex bundle over Deligne-Mumford space that respects certain restriction maps at several level curves in the boundary. We then combine these forms with a certain type of Hamiltonians to get non-positive forms with values in Hamiltonian vector fields, suitable as 0-order perturbations in the Floer equation.

\subsection{Asymptotic markers and cylindrical ends}\label{sec:cylends} 
We will use punctured disks and spheres with a fixed choice of cylindrical end at each puncture. Here, a \emph{cylindrical end} at a puncture is defined to be a biholomorphic identification of a neighborhood of that puncture with one of the following punctured model Riemann surfaces: 
\begin{itemize}
\item Negative interior puncture:
$$
Z^-=(-\infty,0)\times S^1\approx D^2\setminus\{0\}, 
$$
where $D^{2}\subset \C$ is the unit disk in the complex plane.
\item Positive interior puncture:
$$
Z^+=(0,\infty)\times S^1\approx \mathbb{C}\setminus \bar D^2.
$$
\item Negative boundary puncture:
$$ 
\Sigma^-=(-\infty,0)\times[0,1]\approx (D^2\setminus\{0\})\cap H,
$$ 
where $H\subset \C$ denotes the closed upper half plane.
\item Positive boundary puncture:
$$
\Sigma^+=(0,\infty)\times [0,1]\approx(\mathbb{C}\setminus \bar D^2)\cap H.
$$
\end{itemize}
Each of the above model surfaces has a canonical complex coordinate of the form $z=s+it$. Here $s\in\R$ at all punctures, with $s>0$ or $s<0$ according to whether the puncture is positive or negative. At interior punctures, $t\in S^1$ and at boundary punctures, $t\in[0,1]$.

The automorphism group of the cylindrical end at a boundary puncture is $\R$ and the end is thus well-defined up to a contractible choice of automorphisms. For a positive or negative interior puncture, the corresponding automorphism group is $\R\times S^{1}$. Thus the cylindrical end is well-defined up to a choice of automorphism in a space homotopy equivalent to $S^{1}$. To remove the $S^{1}$-ambiguity, we fix an \emph{asymptotic marker} at the puncture, i.e.~a tangent half-line at the puncture, and require that it corresponds to $(0,\infty)\times \{1\}$ or to $(-\infty,0)\times \{1\}$, $1\in S^{1}$, at positive or negative punctures, respectively. The cylindrical end at an interior puncture \emph{with asymptotic marker} is then well-defined up to contractible choice.

We next consider various ways to induce asymptotic markers at interior punctures that we will eventually assemble into a coherent choice of asymptotic markers over the space of punctured spheres and disks. Consider first a disk $D$ with interior punctures and with a distinguished boundary puncture $p$. Then $p$ determines an asymptotic marker at any interior puncture $q$ as follows. There is a unique holomorphic diffeomorphism $\psi\colon D\to D^{2}\subset\C$ with $\psi(q)=0$ and $\psi(p)=1$. Define the asymptotic marker at $q$ in $D$ to correspond to the direction of the real line at $0\in D^{2}$, i.e.~the direction given by the vector $d\psi^{-1}(0)\cdot 1$. See Figure~\ref{fig:markers}.

Similarly, on a sphere $S$, a distinguished interior puncture $p$ with asymptotic marker determines an asymptotic marker at any other interior puncture $q$ as follows. There is a holomorphic map $\psi\colon S\to \R\times S^{1}$ taking $p$ to $\infty$, $q$ to $-\infty$, and the asymptotic marker to the tangent vector of $\R\times\{1\}$. We take the asymptotic marker at $q$ to correspond to the tangent vector of $\R\times\{1\}$ at $-\infty$ under $\psi$. See Figure~\ref{fig:markers}.

\begin{figure}
         \begin{center}
\input{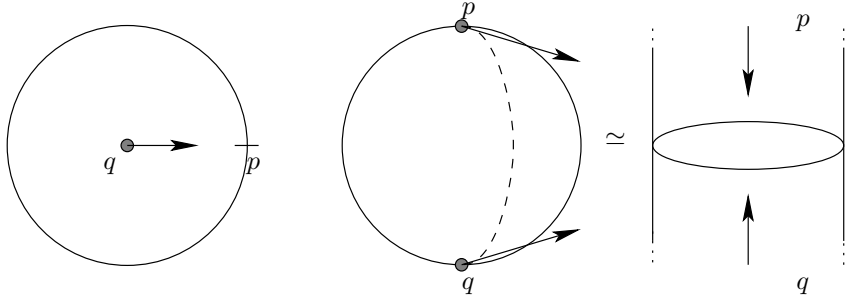}
         \end{center}
\caption{Inducing markers at negative interior punctures.  \label{fig:markers}}
\end{figure}

For a more unified notation below we use the following somewhat involved convention for our spaces of disks and spheres. Let $h\in\{0,1\}$.
For $h=1$ and $m,k\ge 0$, let $\mathcal{D}_{h;hm,k}'=\mathcal{D}_{1;m,k}'$ denote the moduli space of disks with one positive boundary puncture, $m\ge 0$ negative boundary punctures, 
and $k$ negative interior punctures. For $h=0$ and $k\ge 0$, let $\mathcal{D}_{h;hm,k}'=\mathcal{D}_{0;0,k}'$ denote the moduli space of spheres with one positive interior puncture with asymptotic marker and $k$ negative interior punctures. 

As explained above there are then, for both $h=0$ and $h=1$, induced asymptotic markers at all the interior negative punctures of any element in $\mathcal{D}_{h;hm,k}'$. The space $\mathcal{D}_{h;hm,k}'$ admits a natural compactification that consists of several level disks and spheres, see~\cite[\S4]{BEHWZ} and also~\cite{KSV}. We introduce the following notation to describe the boundary. Consider a several level curve. We associate to it a downwards oriented rooted tree $\Gamma$ with one vertex for the positive puncture of each component of the several level curve and one edge for each one of the negative punctures of the components of the several level curves. See Figure~\ref{fig:trees} for examples. Here the root of the tree is the positive puncture of the top level curve and the edges attached to it are the edges of the negative punctures in the top level oriented away from the root. The definition of $\Gamma$ is inductive: the vertex of the positive puncture of a curve $C$ in the $j^{\rm th}$ level is attached to the edge of the negative puncture of a curve in the $(j-1)^{\rm th}$ level where it is attached. All edges of negative punctures of $C$ are attached to the vertex of the positive puncture of $C$ and oriented away from it. Then the boundary strata of $\mathcal{D}_{h;hm,k}'$ are in one to one correspondence with such graphs $\Gamma$ and the components of the several level curve are in one-to-one correspondence with downwards oriented sub-trees consisting of one vertex and all edges emanating from it. For example the graph of a curve lying in the interior of $\mathcal{D}_{h;hm;k}'$ is simply a vertex with $hm+k$ edges attached and oriented away from the vertex. To distinguish the edges of such graphs $\Gamma$, we call an edge a \emph{gluing edge} if it is attached to two vertices and \emph{free} if it is attached only to one vertex. 

\begin{figure}
         \begin{center}
\input{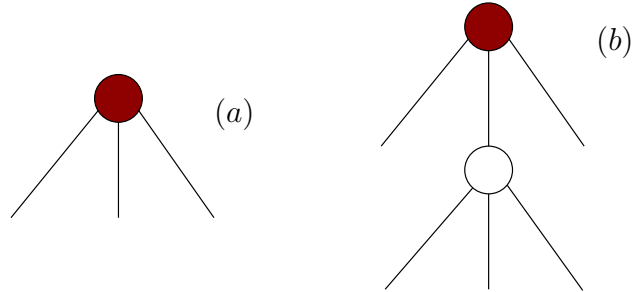}
         \end{center}
\caption{(a) A curve in the main stratum of $\mathcal{D}'_{h;hm,k}$ with $hm+k=3$; (b) A $2$-level curve in the boundary of $\mathcal{D}'_{h;hm,k}$ with $hm+k=5$.  \label{fig:trees}}
\end{figure}

Note next that the induced asymptotic markers are compatible with the level structure in the boundary of $\mathcal{D}_{h;hm,k}'$ in the sense that they vary continuously with the domain inside the compactification. To see this, note that in a boundary stratum corresponding to a graph $\Gamma$, it is sufficient to study neck stretching for cylinders corresponding to linear subgraphs of $\Gamma$, and here the compatibility of asymptotic markers with the level structure is obvious.  

Consider the bundle $\mathcal{C}_{h;hm,k}'\to \mathcal{D}_{h;hm,k}'$, with $h\in\{0,1\}$, $m,k\ge 0$, of disks or spheres with punctures with cylindrical ends compatible with the markers. The fiber of this bundle is contractible so there exists a section. We next show that there is also a section over the compactification of $\mathcal{D}_{h;hm,k}'$. The proof is inductive. We first choose cylindrical ends for disks and spheres with 3 punctures. Gluing these we get cylindrical ends in a neighborhood of the boundary of the moduli space of disks and spheres with 4 punctures. Since the fiber of $\mathcal{C}_{h;hm;k}'$ is contractible this choice can be extended continuously over the whole space of disks and spheres with 4 punctures. The proof then continues inductively in the obvious way: a choice of cylindrical ends for disks and spheres whose number of negative punctures is strictly smaller than $m+k$ determines a section of the bundle $\mathcal{C}_{h;hm,k}'\to\mathcal{D}_{h;hm,k}'$ near its boundary, and this section extends because the fiber of the bundle is contractible. 

Let $\{\mathcal{D}_{h;hm,k}\}_{h\in\{0,1\},\,k,m\ge 0,}$, $\mathcal{D}_{h;hm,k}\colon \mathcal{D}_{h;hm,k}'\to\mathcal{C}_{h;hm,k}'$ denote a system of sections as in the inductive construction above, with $\mathcal{D}_{h;hm,k}$ defined over the compactification of $\mathcal{D}_{h;hm,k}'$. We say that 
$$
\mathcal{D} \quad =\bigcup_{h\in\{0,1\};\,m,k\ge 0}\mathcal{D}_{h;hm,k}
$$ 
is a \emph{system of cylindrical ends that is compatible with breaking}. 

We identify $\mathcal{D}_{h;hm,k}$ with its graph and think of it as a subset of $\mathcal{C}_{h;hm,k}'$. The projection of $\mathcal{D}_{h;hm,k}$ onto $\mathcal{D}_{h;hm,k}'$ is a homeomorphism and, after using smooth approximation, a diffeomorphism with respect to the natural stratification of the space determined by several level curves. Via this projection we endow $\mathcal{D}_{h;hm,k}$ with the structure of a set consisting of (several level) curves with additional data corresponding to a choice of a cylindrical end neighborhood at each puncture. 

A neighborhood of a several level curve $S\in\mathcal{D}_{h;hm,k}$ can then be described as follows. Consider the graph $\Gamma$ determined by $S$. Let $v_0,v_1,\dots,v_{r}$ denote the vertices of $\Gamma$ with $v_0$ the top vertex and let $e_1,\dots,e_s$ denote the gluing edges of $\Gamma$. Let $U_j$ be neighborhoods in $\mathcal{D}_{h_j;h_jm_j,k_j}$ of the component corresponding to $v_j$. Then
a neighborhood $U$ of $S$ is given by 
\begin{equation}\label{eq:bdrycoord}
U=\left(\prod_{\begin{smallmatrix}
v_j \text{ vertex} \\ \text{of }\Gamma \end{smallmatrix}} 
U_{j}\right)
\times \left(\prod_{\begin{smallmatrix}
e_l \text{ gluing} \\ \text{edge of }\Gamma \end{smallmatrix}}
(\rho_{0;l},\infty)\right),
\end{equation}   
$\rho_{0;l}\ge 0$ for $1\le j\le s$.
Here the \emph{gluing parameters} $\rho_l\in (\rho_{0;l},\infty)$ measure the length of the breaking cylinder or strip corresponding to the gluing edge $e_l$. More precisely, assume that $e_l$ connects $v_i$ and $v_j$ and corresponds to the curve $S_{j}$ of $v_j$ attached at its positive puncture $p_j$ to a negative puncture $q_i$ of the curve $S_i$ of $v_i$. Then, given the cylindrical ends $(-\infty,0]\times S^1$ (interior case) or $(-\infty,0]\times [0,1]$ (boundary case) for $q_i$, respectively $[0,\infty)\times S^1$ (interior case) or $[0,\infty)\times [0,1]$ (boundary case) for $p_j$, the glued curve corresponding to the parameter $\rho_l\in(0,\infty)$ is obtained via the gluing operation on these cylindrical ends defined by cutting out $(-\infty,-\rho_l/2)\times S^1$ or $(-\infty,-\rho_l/2)\times [0,1]$ from the cylindrical end of $q_i$, cutting out $(\rho_l/2,\infty)\times S^1$ or $(\rho_l/2,\infty)\times [0,1]$ from the cylindrical end of $p_j$, and gluing the remaining compact domains in the cylindrical ends by identifying $\{-\rho_l/2\}\times S^1$ with $\{\rho_l/2\}\times S^1$, respectively $\{-\rho_l/2\}\times [0,1]$ with $\{\rho_l/2\}\times [0,1]$. We refer to the resulting compact domain as the \emph{breaking cylinder or strip}, and we refer to $\{-\rho_l/2\}\times S^1\equiv \{\rho_l/2\}\times S^1$ or $\{-\rho_l/2\}\times [0,1]\equiv \{\rho_l/2\}\times [0,1]$ as its \emph{middle circle or segment}. Given a several level curve $S$ in this neighborhood we write ${\overline S}_{j}$ for the closures of the components that remain if the middle circle or segment in each breaking cylinder or strip is removed, and that correspond to subsets of the levels $S_{j}$ of the broken curve.
See Figure~\ref{fig:gluing}.

\begin{figure}
         \begin{center}
\input{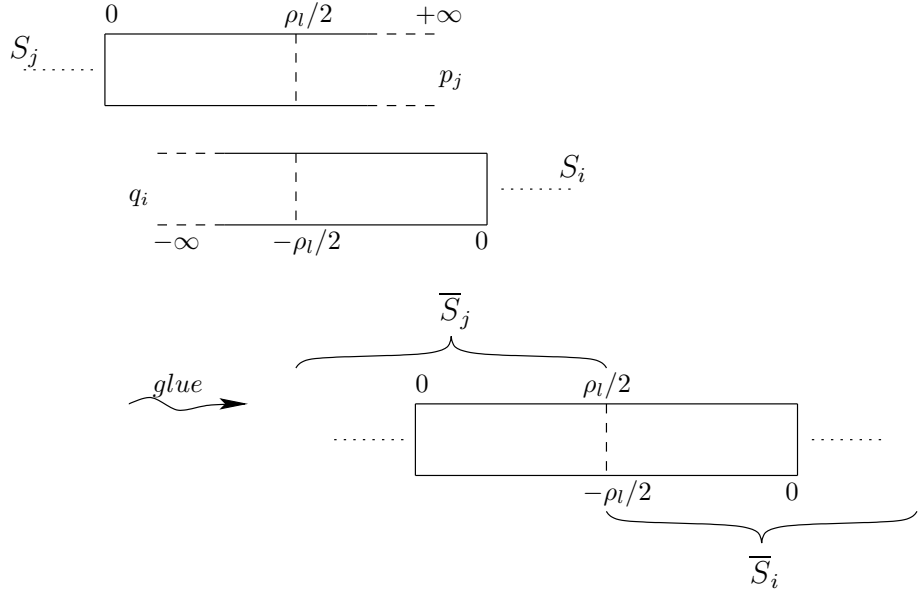}
         \end{center}
\caption{Gluing of a nodal curve in cylindrical coordinates.  \label{fig:gluing}}
\end{figure}

\subsection{Almost complex structures}\label{sec:acs}
We next introduce splitting compatible families of almost complex structures over $\mathcal{D}$. Let $\mathcal{J}(X)$ denote the space of almost complex structures on $X$ compatible with $\omega$ and adapted to the contact form $\alpha$ in the cylindrical end, i.e.~if $J\in\mathcal{J}$ then in the cylindrical end $J$ preserves the contact planes and takes the vertical direction to the Reeb direction. Our construction of a family of almost complex structures is inductive. We start with strips, cylinders and cylinders with slits with coordinates $s+it$. Here we require that $J=J_{t}$ depends only on the $I$ or $I/\partial{I}$ coordinate. Assume that we have defined a family of almost complex structures $J_{z}$ for all curves $D_{h;hm,k}$, $hm+k\le p$ which have the form above in every cylindrical end and which commute with restriction to components for several level curves. By gluing we then have a field of almost complex structures in a neighborhood of the boundary of $\mathcal{D}_{h;hm,k}$ for $hm+k=p+1$. Since $\mathcal{J}$ is contractible it is easy to see that we can extend this family to all of $\mathcal{D}_{h;hm,k}$. We call the resulting family of almost complex structures over the universal curve corresponding to $\mathcal{D}$, \emph{splitting compatible}.

\subsection{A simplex bundle}\label{sec:simplexbundle}
Consider the trivial bundle 
\[
\mathcal{E}^{hm+k-1}=\mathcal{D}_{h;hm,k}\times\Delta^{hm+k-1} \to\mathcal{D}_{h;hm,k}
\] 
over $\mathcal{D}_{h;hm,k}$, with fiber the open $(hm+k-1)$-simplex 
$$
\Delta^{hm+k-1}=\left\{(s_1,\dots,s_{hm+k})\colon \sum_i s_i = 1,\ s_i>0\right\}.
$$
Since the bundle is trivial, it extends as such over the compactification of $\mathcal{D}_{h;hm,k}$. We think of the coordinates of a point $(s_1,\dots,s_{hm+k})\in \Delta^{hm+k-1}$ over a disk or sphere $D_{h;hm,k}\in \mathcal{D}_{h;hm,k}$ as representing weights at its negative punctures, and we think of the positive puncture as carrying the weight $1$. 

We next define restriction maps for $\mathcal{E}^{hm+k-1}$ over the boundary of $\mathcal{D}_{h;hm,k}$. Let $s=(s_1,\dots,s_{hm+k})\in\Delta^{hm+k-1}$ denote the weights of a several level curve $S$ in the boundary of $\mathcal{D}_{h;hm,k}$ with graph $\Gamma$.  Let $S_j$ be a component of this building corresponding to the vertex $v_j$ of $\Gamma$, with positive puncture $q_0$ and negative punctures $q_1,\dots,q_n$. Define the weight $w(q_l)$ at $q_l$, $l=0,\dots,n$ as follows. For $l=0$, $w(q_0)$ equals the sum of all weights at negative punctures $q$ of the total several level curve for which there exists a level-increasing path in $\Gamma$ from $v_j$ to $q$. For $l\ge 1$, if the edge of the negative puncture $q_l$ is free then $w(q_l)$ equals the weight of the puncture $q_l$ as a puncture of the total several level curve, and if the edge is a gluing edge connecting $v_j$ and $v_t$, then $w(q_l)$ equals the sum of all weights at negative punctures $q$ of the total several level curve for which there exists a level-increasing path in $\Gamma$ from $v_t$ to $q$. Note that $w(q_0)=w(q_1)+\dots+w(q_n)$ by construction.

The \emph{component restriction map} $r_{j}$ then takes the point $s\in\Delta^{hm+k-1}$ over $S$ to the point 
\[ 
r_{j}(s)=\tfrac{1}{w(q_0)}(w(q_1),\dots,w(q_n))\in\Delta^{n-1}
\]
over $S_j$ in $\mathcal{E}^{n-1}$. The component restriction map $r_{j}$ is defined on the restriction of $\mathcal{E}^{hm+k-1}$ to the stratum that corresponds to $\Gamma$ in the boundary of $\mathcal{D}_{h;hm,k}$.

\begin{figure}
         \begin{center}
\input{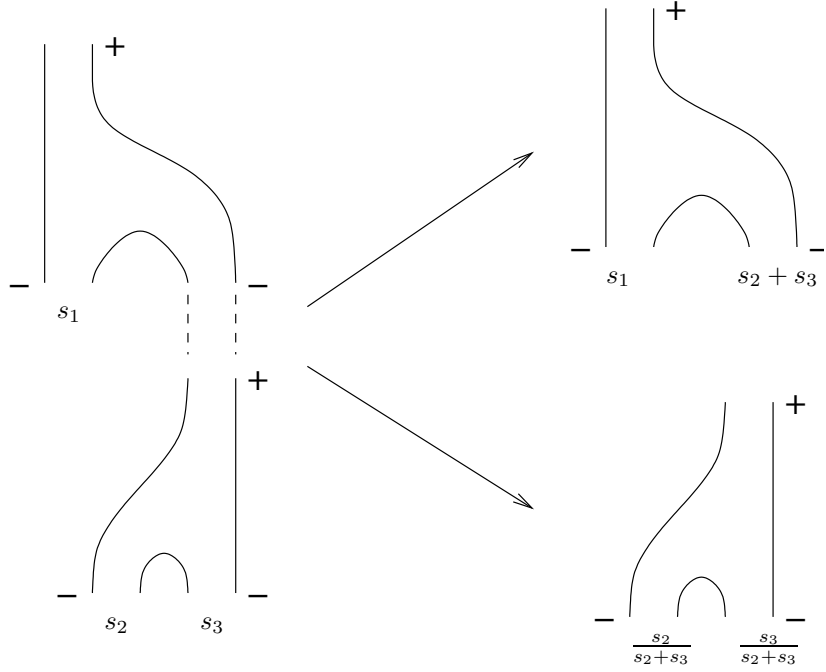}
         \end{center}
\caption{Component restriction maps.  \label{fig:weights}}
\end{figure}

\subsection{Superharmonic functions and non-positive 1 forms}\label{sec:superharmonic}
Our main Floer homological constructions involve studying Floer holomorphic curves parameterized by finite dimensional families of 1-forms with values in Hamiltonian vector fields. As discussed in Section \ref{S:intr}, it is important that the $1$-forms are non-positive, i.e.~the associated 2-forms are non-positive multiples of the area form. Furthermore, in order to derive basic homological algebra equations, the 1-forms must be gluing/breaking compatible on the boundary of Deligne-Mumford space. In this section we construct a family of superharmonic functions parameterized by $\mathcal{E}$ that is compatible with the component restriction maps at several level curves. The differentials of these functions multiplied by the complex unit $i$ then give a family of 1-forms with non-positive exterior derivative that constitutes the basis for our construction of 0-order term in the Floer equation. 

Fix a smooth decreasing function $\kappa\colon(0,1]\to [0,\infty)$ such that $\kappa(1)=0$ and 
\begin{equation}\label{eq:defrho}
\lim_{s\to 0+}\kappa(s)=+\infty.
\end{equation}
We will refer to $\kappa$ as a \emph{stretching profile}.

We will construct a family of functions over curves in $\mathcal{D}$ parameterized by the bundle $\mathcal{E}$ in the following sense. If $e\in\mathcal{E}$ belongs to the fiber over a one level curve $D_{h;hm,k}\in\mathcal{D}_{h;hm,k}$ then $g_{e}\colon D_{h;hm,k}\to\R$. If $D_{h;hm,k}$ is a several level curve with graph $\Gamma$ and components $S_j$ corresponding to its vertices $v_j$, $j=0,\dots,s$, then $g_{e}$ is the collection of functions $g_{r_0(e)},\dots,g_{r_s(e)}$ on $S_0,\dots, S_s$, where $r_{j}$ denotes the component restriction map to $S_{j}$. Our construction uses induction on the number of negative punctures and on the number of levels.

In the first case $hm+k=1$ and the domain is the strip $\R\times[0,1]$, the cylinder $\mathbb{R}\times S^{1}$, or the cylinder with a slit (which we view as a subset of $\R\times S^{1}$). Over these domains the fiber of $\mathcal{E}$ is a point $e$ and we take the function $g_{e}$ to be the projection to the $\R$-factor.  

For $hm+k>1$, we specify properties of the functions separately for one level curves in the interior of $\mathcal{D}_{h;hm,k}$ and for a neighborhood of several level curves near the boundary.  We start with one level curves. Let $e$ be a section of $\mathcal{E}$ over one level curves in the interior $\mathring{\mathcal{D}}_{h;hm,k}$. Let $D_{h;hm,k}\in\mathring{\mathcal{D}}_{h;hm,k}$ and write $e=(w_1,\dots,w_{hm+k})\in{\Delta}^{hm+k-1}$.

We say that a smooth family of functions $g_{e}$ over the interior satisfies the  
\emph{one level conditions} if the following hold. (We write $\pi\colon \mathcal{E}\to\mathcal{D}$ for the projection.)
\begin{itemize}
\item[$(\mathbf{I})$]
There is a constant $c_0=c_0(\pi(e))$ such that in a neighborhood of infinity in the cylindrical end at the positive puncture 
\begin{equation}\label{eq:gcylend+}
g_{e}(s+it)=c_0+s,
\end{equation}
where $s+it$ is the complex coordinate in the cylindrical end, i.e.~in $[0,\infty)\times S^{1}$ for an interior puncture and in $[0,\infty)\times[0,1]$ for a boundary puncture, see Section \ref{sec:cylends}.
\item[$(\mathbf{II}_{1})$]
There are constants $\sigma=\sigma(\pi(e))\in [1,2)$, $R=R(\pi(e))>0$, $c_j=c_j(e)$, and $c_j'=c_j'(e)$ for $j=1,\dots,hm+k$, such that in a neighborhood of infinity in the cylindrical end of the $j^{{\rm th}}$ negative puncture of the form $(-\infty,0]\times S^1$ for interior punctures or $(-\infty,0]\times [0,1]$ for boundary punctures, we have $g_e(s+it)=g_e(s)$, where
\begin{equation}
g_{e}(s)=
\begin{cases}\label{eq:gcylend-1}
c'_j+\sigma w_js &\text{ for } -R\ge s \ge -R-\kappa(w_j),\\
c_j+s &\text{ for }  -R-\kappa(w_j)-1 \ge s > -\infty,
\end{cases}
\end{equation}
is a concave function, $g''_e(s)\le 0$, and where $\kappa$ is the stretching profile \eqref{eq:defrho}. In particular, for each weight $w_j$ at a negative puncture there is a cylinder or strip region of length at least $\kappa(w_j)$ along which $g_{e}(s+it)=\epsilon s + C$, with $0<\epsilon\le 2w_j$.
\item[$(\mathbf{III})$] The function is superharmonic, $\Delta g_{e}\le 0$ everywhere. 
\item[$(\mathbf{IV})$] The derivative of $g_e$ in the direction of the normal $\nu$ of the boundary $\partial D_{h;hm,k}$ vanishes everywhere: 
\[
\frac{\partial g_{e}}{\partial\nu}=0\quad\text{ along }\quad\partial D_{h;hm,k}.
\]
\end{itemize}

\begin{rmk} For the boundary condition $(\mathbf{IV})$, note that, for the cylinder with a slit, in local coordinates $u+iv$, $v\ge 0$ at the end of the slit the standard function looks like $g_{e}(u+iv)=u^{2}-v^{2}$, and $\frac{\partial g_{e}}{\partial v}=0$.
\end{rmk}

\begin{rmk}
The appearance of the ``extra factor'' $\sigma$ in~\eqref{eq:gcylend-1} is to allow for a certain interpolation below. As we shall see, we can take $\sigma$ arbitrarily close to $1$ on compact sets of $\mathring{\mathcal{D}}_{h;hm,k}$. As mentioned in Section \ref{S:intr}, one of the main uses of weights is to force solutions to degenerate for small weights, and for desired degenerations it is enough that $\sigma$ be uniformly bounded. At the opposite end we find the following restriction on $\sigma$: superharmonicity in the cylindrical end near a negative puncture where the weight is $w_j$ implies that $\sigma w_j\le 1$, and in particular $\sigma\to 1$ if $w_j\to 1$. In general, 
superharmonicity of the function $g_e$ is equivalent to the differential $d(-i^*dg_e)$ being non-positive with respect to the conformal area form on the domain $D_{h;hm,k}$. This is compatible with Stokes' theorem, which gives
$$
\int_{D_{h;hm,k}}-d(i^*dg_e)= 1-(hm+k)\le 0.
$$
\end{rmk}

We will next construct families of functions satisfying the one-level condition over any compact subset of the interior of $\mathcal{D}_{h;hm,k}$. Later we will cover all of $\mathcal{D}_{h;hm,k}$ with a system of neighborhoods of the boundary where condition $\mathbf{II}_{1}$ above is somewhat weakened but still strong enough to ensure degeneration for small weights. 

\begin{lemma}\label{lem:1-supharm}
If $e\colon \mathring{\mathcal{D}}_{h;hm,k}\to \mathcal{E}$ is a constant section, then, over any compact subset $\mathcal{K}\subset\mathring{\mathcal{D}}_{h;hm,k}$, there is a family of functions $g_{e}$ that satisfies the one level conditions. Moreover, we can take $\sigma$ in $(\mathbf{II}_{1})$ arbitrarily close to $1$.
\end{lemma}

\begin{proof}
For simpler notation, let $D=D_{h;hm,k}$.
Consider first the case when the positive puncture $p$ and all the negative punctures $q_1,\dots,q_{k}$ are interior. Fix an additional marked point in the domain. For each $q_{j}$, fix a conformal map to $\R\times S^{1}$ which takes the positive puncture to $\infty$, the marked point to some point in $\{0\}\times S^{1}$, and the negative  puncture to $-\infty$. Fix $\sigma\in (1,2)$ and let $g_{j}'\colon D\to\R$ be the function $g_j'=\frac {1+\sigma} 2 w_js_j+c_j$ with $s_j$ the $\mathbb{R}$-coordinate on $\mathbb{R}\times S^1$. Let $g_j$ be a concave approximation of this function 
with second derivative non-zero only on two intervals of finite length located near $\pm\infty$, linear of slope $w_j$ near $+\infty$ and linear of slope $\sigma w_j$ near $-\infty$, see Figure~\ref{fig:g}. 
Consider the function
\[ 
g=\sum_{j=1}^{k} g_j.
\]
Then $g$ is superharmonic but it does not quite have the right behavior at the punctures. 
Here however, the leading terms are correct and the errors are exponentially small. We turn off the exponential error in a neighborhood of $q_j$ in the region of support of the second derivative of $g_{j}$. We can arrange the parameters so that the resulting function satisfies \eqref{eq:gcylend+} near the positive puncture, and it satisfies the top equation in the right hand side of \eqref{eq:gcylend-1} in some neighborhood of $q_{j}$. In order to achieve the bottom equation in a neighborhood of $q_j$ we simply replace the linear function of slope $\sigma w_j$ by a concave function that interpolates between it and the linear function of slope $1$. The fact that we can take $\sigma$ arbitrarily close to $1$ follows from the construction.

The case of boundary punctures can be treated in exactly the same way. In case of a positive boundary puncture and a negative interior puncture we replace the cylinder above with the cylinder with a slit along $[0,\infty)\times \{1\}$ and in case of both positive and negative boundary punctures we use the cylinder with a slit all along $\R\times\{1\}$. 
\end{proof}

\begin{remark}
For future reference we call the regions in the cylindrical ends where $\Delta g_{e}<0$ \emph{regions of concavity}.
\end{remark}

\begin{figure}
         \begin{center}
\input{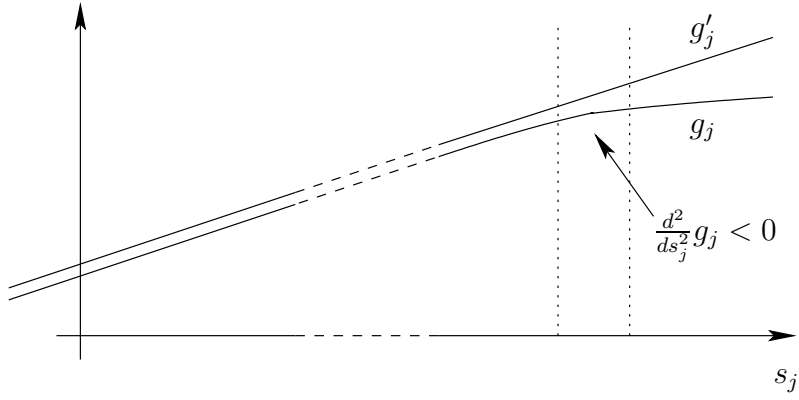}
         \end{center}
\caption{A function $g_j$ that is strictly concave on the region of concavity near $+\infty$.}
\label{fig:g}
\end{figure} 

We next want to define a corresponding notion for several level curves. To this end we consider nested neighborhoods
\[ 
\dots\subset\mathcal{N}^{\ell}\subset \mathcal{N}^{\ell-1}\subset \mathcal{N}^{\ell-2} \subset \dots \subset \mathcal{N}^{2},
\]
were $\mathcal{N}^{j}$ is a neighborhood of the subset $\mathcal{D}^{j}\subset\mathcal{D}$ of $j$-level curves. Consider constant sections $e$ of $\mathcal{E}^{hm+k-1}$ over $\mathring{\mathcal{D}}_{h;hm,k}$ and let $g_{e}$ be a family of functions. The \emph{$\ell$-level conditions} are the same as the one level conditions $\mathbf{I}$, $\mathbf{III}$, and $\mathbf{IV}$, and also the following new condition: 
\begin{itemize}
\item[$(\mathbf{II}_{\ell})$] 
For curves in $\mathcal{N}^{\ell}-\mathcal{N}^{\ell-1}$ with $e=(w_1,\dots,w_{hm+k})$ and any $j$, there is a strip or cylinder region of length at least $\kappa((w_{j})^{\frac{1}{\ell}})$ where $g_{e}(s+it)=\epsilon s + C$ for $0<\epsilon\le 2(w_{j})^{\frac{1}{\ell}}$.  
\end{itemize}

Our next lemma shows that there is a family of functions $g_{e}$ that satisfies the $\ell$-level condition and that is also compatible with splittings into several level curves in the following sense. 

We say that a family of functions $g_e$ as above is \emph{splitting compatible} if the following holds. 
If $S_{\nu}\in\mathring{\mathcal{D}}_{h;hm,k}$, $\nu=1,2,3,\dots$ is a family of curves that converges as $\nu\to\infty$ to an $\ell$-level curve with components $S_{0},\dots, S_{m}$ and if $K_{\nu}\subset S_{\nu}$ is any compact subset that converges to a compact subset $K_j$ of $S_{j}$, then there is a sequence of constants $c_{\nu}$ such that the restriction $g_{e}|_{K_{\nu}}+c_{\nu}$ converges to $g_{r_j(e_j)}|_{K_j}$, where $r_j(e)$ is the component restriction of $e$ to $S_{j}$. 

\begin{lemma}\label{lemma:ell-levelsupharm}
There exists a system of neighborhoods 
\[ 
\dots\subset\mathcal{N}^{\ell}\subset \mathcal{N}^{\ell-1}\subset \mathcal{N}^{\ell-2} \subset \dots \subset \mathcal{N}^{2},
\]
and a splitting compatible family of functions $g_{e}$ parameterized by constant sections of $\mathcal{E}$ that satisfies the $\ell$-level condition for all $\ell\ge 1$. 
\end{lemma}

\begin{proof}
The proof is inductive. In the first case $hm+k=2$ there are only one level curves and we use the canonical functions $g_{e}$ discussed above. Consider next a gluing compatible section $e$ over $\mathcal{D}_{h;hm,k}$ with $hm+k=3$. This space is an interval and the boundary points correspond to two level curves $S$ with both levels $S_{0}$ and $S_{1}$ in $\mathcal{D}_{h;hm,k}$, $hm+k=2$. Consider a neighborhood of such a two level curve in $\mathcal{D}_{h;hm,k}$ parameterized by a gluing parameter $\rho\in[0,\infty)$, see \eqref{eq:bdrycoord}. Assume that the positive puncture of $S_{1}$ is attached at a negative puncture of $S_{0}$. Write $S(\rho)\in\mathcal{D}_{h;hm,k}$, $hm+k=3$ for the resulting domain and write $S_{j}(\rho)$ for the part of the curve $S(\rho)$ that is naturally a subset of $S_{j}$, see the discussion in Section~\ref{sec:cylends}. Let $g_{r_0(e)}$ and $g_{r_1(e)}$ denote the functions of the component restrictions of $e$ to $S_{0}$ and $S_{1}$. Then there is a constant $c(\rho)$ such that
\begin{equation}\label{eq:supharmcomp} 
c(\rho)=g_{r_0(e)}|_{\partial S_{0}(\rho)}-g_{r_1(e)}|_{\partial S_{1}(\rho)}.
\end{equation} 
We then define the function $g_{e}(\rho)\colon S(\rho)\to\R$ as
\[ 
g_{e}(\rho)=
\begin{cases}
g_{r_0(e)} &\text{ on }S_{0}(\rho),\\
c(\rho)+g_{r_1(e)} &\text{ on }S_{1}(\rho).
\end{cases}
\]
Then $g_{e}(\rho)$ is smooth and satisfies $(\mathbf{I})$, $(\mathbf{III})$, and $(\mathbf{IV})$ and has the required properties for restrictions to levels. Furthermore, the restriction of $g_{e}(\rho)$ to $S_{0}(\rho)$ satisfies \eqref{eq:gcylend-1} with $\sigma w_j$ replaced by $\sigma w(q_0)$, where $w(q_0)$ is the weight of $r_0(e)$ at the negative puncture $q_0$ of $S_{0}$ where $S_1$ is attached (except that the interval in the second equation is not infinite but finite) and the restriction of $g_{e}(\rho)$ to $S_{1}(\rho)$ satisfies \eqref{eq:gcylend-1} with the weights of $r_1(e)$ at the negative ends of $S_{1}$. 
Let $w_j(r_1(e))$ denote the weights at the negative punctures $q_j$ of $S$ which are negative punctures of $S_1$, seen as negative punctures of $S_1$. Then by definition 
$$
w_j=w(q_0)w_j(r_1(e)).
$$
Since 
$$
(w_j)^{\frac 1 2}=(w(q_0)w_j(r_1(e)))^{\frac 1 2}\ge \min(w(q_0),w_j(r_1(e)))
$$
and
$$
\kappa((w_j)^{\frac 1 2})\le \kappa(\min(w(q_0),w_j(r_1(e)))=\max(\kappa(w(q_0)),\kappa(w_j(r_1(e)))),
$$
we find that there exists a strip or cylinder region of length at least $\kappa((w_j)^{\frac 1 2})$ where $g_e(s+it)=\epsilon s+C$, with $0<\epsilon\le 2(w_j)^{\frac 12}$. Thus the two level condition $(\mathbf{II}_{2})$ holds. 


We next want to extend the family of functions over all of $\mathcal{D}_{h;hm,k}$, $hm+k=3$, respecting condition $(\mathbf{II}_{2})$. To this end we consider a neighborhood ${\mathcal{N}^{2}}'$ of the broken curves in the boundary where the glued functions described above are defined. Using the gluing parameter this neighborhood can be identified with a half infinite interval. As the gluing parameter decreases we deform the derivative of the function as follows: we decrease it uniformly below the gluing region and stretch the region near the negative puncture where it is small, until we reach the level one function. See Figure~\ref{fig:necked-weights}. For this family $g_{e}$ conditions $(\mathbf{I})$, $(\mathbf{II}_{2})$, $(\mathbf{III})$, and  $(\mathbf{IV})$ hold everywhere and $(\mathbf{II}_{1})$ holds in the compact subset of $\mathring{\mathcal{D}}_{h;hm,k}$ which is the complement of a suitable subset $\mathcal{N}^{2}\subset {\mathcal{N}^{2}}'$. 

For more general two level curves with $hm+k>3$ lying in $\mathcal{N}_{2}-\mathcal{N}_{3}$ we argue in exactly the same way using the gluing parameter to interpolate between the natural gluing of the functions of the component restrictions of $e$ and the function of $e$, see Lemma \ref{lem:1-supharm}, satisfying the one-level condition.

Consider next the general case. Assume that we have found a family of functions $g_{e}$, associated to a constant section $e$ defined over the subset $\mathcal{D}^{\ell}$ consisting of all curves in $\mathcal{D}$ with at most $\ell$ levels, that satisfies conditions $(\mathbf{I})$, $(\mathbf{III})$, and $(\mathbf{IV})$ everywhere, and assume that there are nested neighborhoods
\[ 
\mathcal{N}^{\ell}\subset \mathcal{N}^{\ell-1}\subset \mathcal{N}^{\ell-2} \subset \dots \subset \mathcal{N}^{2},
\]
were $\mathcal{N}^{j}$ is a neighborhood of $\mathcal{D}^{j}$ in $\mathcal{D}^{\ell}$ such that condition $(\mathbf{II}_{j})$ holds in $\mathcal{N}^{j}-\mathcal{N}^{j-1}$. 

Consider a curve $S$ in the boundary of $\mathcal{D}_{h;hm,k}$ with $\ell+1$ levels. Assume that the top level curve $S_{0}$ of $S$ has $r$ negative punctures at which there are curves $S_{1},\dots,S_{r}$ of levels $\le \ell$ attached. Let $r_j(e)$ denote the component restriction $S_{j}$, $j=0,1,\dots,r$.
Our inductive assumption gives a smooth family of superharmonic functions with properties $(\mathbf{I})$, $(\mathbf{III})$, and $(\mathbf{IV})$ for curves in a neighborhood of these broken configurations depending smoothly on $r_j(e)$. Denote the corresponding functions $g_{r(e_j)}\colon S_{j}\to\R$. Consider now a coordinate neighborhood $U$ of the form~\eqref{eq:bdrycoord} around $S$:
\[
U=U^{0}\times \prod_{i=1}^{r} (0,\infty)_{j}\times U^{j}.
\] 
Let $\rho=(\rho_1,\dots,\rho_{r})$. For curves $S_{j}\in U^{j}$, write $S(\rho)$ for the curve that results from gluing these according to $\rho$ and in analogy with the two level case, write $S_{j}(\rho)$ for the part of $S(\rho)$ that is naturally a subset of $S_{j}$. Our inductive assumption then shows that there are constants $c_{j}(r_0(e),r_j(e),\rho_j)$, $j=1,\dots,r$, such that
\begin{equation}\label{eq:supharmcomp2}
c_{j}(r_0(e),r_j(e),\rho_j)=g_{r_0(e)}^{0}|_{\partial_{j} S_{0}(\rho)}-g_{r_j(e)}|_{\partial S_{j}(\rho)},
\end{equation}
where $\partial_{j} S_{0}(\rho)$ is the boundary component of $S_{0}(\rho)$ where $S_{j}(\rho)$ is attached. Define the function $g_{e}(\rho)\colon S(\rho)\to\R$ as
\[ 
g_{e}(\rho)=
\begin{cases}
g_{r_0(e)} &\text{ on } S_{0}(\rho),\\
g_{r_j(e)}+c_j(r_0(e),r_j(e),\rho_j) &\text{ on } S_{j}(\rho), \quad j=1,\dots,r.
\end{cases}
\]
It is immediate that the function $g_{e}(\rho)$ satisfies $(\mathbf{I})$, $(\mathbf{III})$, and $(\mathbf{IV})$. We show that condition $(\mathbf{II}_{\ell+1})$ holds. Let $q$ be a negative puncture in some $S_j$, $j=1,\dots,r$. Let
$$
w_q^{j\prime}=w_j^0 w_q^j,
$$
where $w_j^0$ is the weight of $r_0(e)$ at the negative puncture of $S_0$ where $S_j$ is attached and where $w_q^{j}$ is the weight of $r_j(e)$ at the negative puncture $q$. Then $w_q^{j\prime}$ is the weight of the puncture $q$ seen as a negative puncture of $S$. Since
$$
(w_q^{j\prime})^{\frac{1}{\ell+1}}=(w_j^0w_q^j)^{\frac{1}{\ell +1}}\ge \min(w_j^0,(w_q^j)^{\frac 1 \ell})
$$ 
and 
$$
\kappa((w_q^{j\prime})^{\frac{1}{\ell+1}})\le \kappa(\min(w_j^0,(w_q^j)^{\frac 1 \ell}))=\max(\kappa(w_j^0),\kappa((w_q^j)^{\frac 1 \ell})),
$$
we deduce that condition $(\mathbf{II}_{\ell+1})$ holds.

This defines $g_{e}(\rho)$ in a collar neighborhood of the boundary of $\mathcal{D}^{\ell+1}$. As in the two level case above we get a family $g_{e}'$ on the complement of half the collar neighborhood, and then by interpolation we obtain a gluing compatible family over all of $\mathcal{D}^{\ell+1}$ that satisfies conditions $(\mathbf{II}_{\ell})$ and $(\mathbf{II}_{\ell+1})$ with respect to an appropriate neighborhood $\mathcal{N}^{\ell+1}$, as required. 
\end{proof}

\begin{figure}
         \begin{center}
\input{necked-weights.pstex_t}
         \end{center}
\caption{$1$-forms with ``necked-weights" and their behavior under gluing.  \label{fig:necked-weights}}
\end{figure}

Using the splitting compatible family of subharmonic functions parameterized by $\mathcal{E}$, we define two families of non-positive 1-forms on the domains in $\mathcal{D}$, likewise parameterized by $\mathcal{E}$. The most basic family is defined as follows. Let $i$ denote the complex structure on the domain $D_{h;hm,k}$ and define
\begin{equation}\label{eq:beta_e}
\beta_{e}=-i^*dg_{e}=\frac{\partial g_{e}}{\partial\sigma}d\tau -\frac{\partial g_{e}}{\partial\tau}d\sigma,
\end{equation}
where $\sigma+i\tau$ is a complex coordinate on $D_{h;hm,k}$. Then we find that
\[
d\beta_{e}=(\Delta g_{e}) d\sigma\wedge d\tau\le 0,
\]
with strict inequality in regions of concavity.

\subsection{Hamiltonians}
We consider two types of Hamiltonians: one for defining the Hamiltonian simplex DGA that we call \emph{one step Hamiltonian} and one for defining cobordism maps between DGAs that we call \emph{two step Hamiltonian}. We use the following convention: if $H\colon X\to\R$ is a Hamiltonian function then we define the corresponding Hamiltonian vector field $X_{H}$ by 
\[
\omega(X_H,\cdot)=-dH.
\]

Let $(X,L)$ be a Weinstein pair with end $[0,\infty)\times (Y,\Lambda)$, and recall our notation $r=e^t$, where $t$ is the coordinate on the factor $[0,\infty)$. We first consider time independent \emph{one step} Hamiltonians $H\colon X\to\R$. Such a function has the following properties:
\begin{itemize}
\item For small $\epsilon>0$, $\frac{\epsilon}{2}\le H \le\epsilon$ and $H$ is a Morse function on the compact manifold with boundary $X\setminus (0,\infty)\times Y$.
\item On $[0,\infty)\times Y$, $H(r,y)=h(r)$ is a function of $r$ only with $h'(r)>0$ and $h''(r)\ge 0$ such that for $r\ge 1$, $H(r)=ar+b$ where $a>0$ and $b$ are real constants. We require that $a$ is distinct from the length of any closed Reeb orbit or of any Reeb chord with endpoints on $\Lambda$.
\end{itemize}

Note that in the symplectization part, where $H=h(r)$, the Hamiltonian vector field is proportional to the Reeb vector field $R$ of the contact form $\alpha$ on $Y$: 
\[
X_H=h'(r)R.
\]

Consider the time 1 flow of the Hamiltonian vector field $X_H$ of $H$. Hamiltonian chords and orbits then come in two classes. Low energy orbits that correspond to critical points of $H$ that we take to lie off of $L$ and low energy chords that correspond to critical points of $H|_{L}$. The low energy chords and orbits are generically transverse. High energy orbits and chords are re-parameterizations of Reeb chords and orbits. The chords are generically transverse but the orbits are generically transverse only in the directions transverse to the orbit but not along the orbit. Following \cite{CFHW}, we pick a small positive time dependent perturbation of $H$ near each orbit based on a Morse function on the orbit that gives two orbits of the time dependent Hamiltonian corresponding to $H$. We call the resulting Hamiltonian a time dependent one step Hamiltonian.

Let $(X_0,L_0)$ be a Weinstein pair with end $[0,\infty)\times (Y_0,\Lambda_0)$ and consider a symplectic cobordism $(X_{10},L_{10})$ with negative end $(Y_0,\Lambda_{0})$ and positive end $(Y_{1},\Lambda_{1})$. Gluing $(X_{10},L_{10})$ to $(X_{0},L_{0})$ we build a new Weinstein manifold $(X_{1},L_{1})$ which contains the compact part of $(X_{0},L_{0})$, connected via $[-R,0]\times(Y_{0},\Lambda_{0})$ to a compact version $(X_{10}',L_{10}')$ of the cobordism, and finally its cylindrical end. Consider time independent \emph{two step} Hamiltonians $H\colon X_{1}\to\R$. Such functions have the following properties: 
\begin{itemize}
\item For small $\epsilon>0$, $\frac{\epsilon}{2}\le H \le\epsilon$ and $H$ is a Morse function on $X_0'$, the complement of $[-R,0]\times Y_0$ in the compact part of $X_0$.
\item On $[-R,-1]\times Y_{0}$, $H(r,y)=h(r)$ is a function of $r$ only with $h'(r)>0$ and $h''(r)\ge 0$ such that for $r\ge -R+1$, $H(r)=ar+b$ where $a>0$ and $b$ are real constants. 
We require that $a$ is distinct from the length of any closed Reeb orbit or Reeb chord with endpoints on $\Lambda_0$ in $Y_0$.
\item On $[-1,0]\times Y_{0}$, $h'(r)\le 0$ and the function becomes constant.
\item Over $X_{10}'$ the function is an approximately constant Morse function.
\item Finally in the positive end the function has the standard affine form of a one step Hamiltonian. 
\end{itemize}

Let $H_1$ be a time dependent one step Hamiltonian on $X_1$ and let $H_0$ be a two step Hamiltonian on $X_1$ with respect to the cobordism $X_{01}$ such that $H_{0}\ge H_{1}$.

\begin{figure}
         \begin{center}
\input{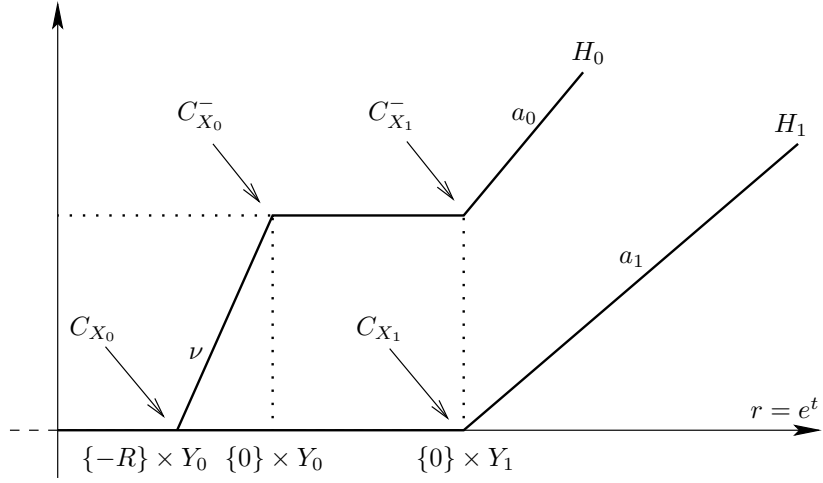}
         \end{center}
\caption{Hamiltonians for cobordism map.  \label{fig:H-cob}}
\end{figure}

We consider chords and orbits of both Hamiltonians. The action of a chord or orbit $\gamma\colon [0,1]\to X$ of $H_j$ is
\[ 
\mathfrak{a}(\gamma)=\int_{0}^{1}\gamma^{\ast}\lambda -\int_{0}^{1}H_j(\gamma(t))\,dt. 
\]
The non-positivity of our 1-forms implies that, if $D_{1;m,k}\in\mathcal{D}_{1;m,k}$ and $u\colon (D_{1;m,k},\partial D_{1;m,k})\to (X_1,L_1)$ lies in the space of solutions of the Floer equation $\mathcal{F}(a;\mathbf{b},\bo{\eta})$ as defined in Section~\ref{sec:Floermodulispaces} below, with $a$ a chord, $\mathbf{b}=b_1\dots b_m$ a word in chords, and $\bo{\eta}=\eta_1\dots\eta_{k}$ a word in periodic orbits, then
\[ 
\mathfrak{a}(a)-(\mathfrak{a}(b_1)+\dots+\mathfrak{a}(b_m))-(\mathfrak{a}(\eta_1)+\dots+\mathfrak{a}(\eta_{k}))\ge 0.
\]
Likewise if $u\in\mathcal{F}(\gamma,\bo{\eta})$ as defined in Section~\ref{sec:Floermodulispaces} below, with $\gamma$ a periodic orbit and $\bo{\eta}=\eta_1\dots\eta_{k}$ a word in periodic orbits, then
\[
\mathfrak{a}(\gamma)-(\mathfrak{a}(\eta_1)+\dots+\mathfrak{a}(\eta_{k}))\ge 0.
\]

\begin{lemma}\label{lemma:actiondecomposition}
The Hamiltonian chords and orbits of $H_1$ decompose into the following subsets:
\begin{itemize}
\item $O_{X_1}$, the chords and orbits that correspond to critical points of $H_1$ in $X_1$. If $\gamma\in O_{X_1}$ then $\mathfrak{a}(\gamma)\approx 0$.
\item $C_{X_1}$, Hamiltonian chords and orbits located near $\{0\}\times Y_1$, and corresponding to Reeb chord and orbits in $(Y_1,\alpha_1)$. If $\gamma\in C_{X_1}$ then $\mathfrak{a}(\gamma)>0$.
\end{itemize}

The Hamiltonian chords and orbits of $H_0$ decompose into the following subsets:
\begin{itemize}
\item $O_{X_0}$, the chords and orbits that correspond to critical points of $H_0$ in $X_0$. If $\gamma\in O_{X_0}$ then $\mathfrak{a}(\gamma)\approx 0$.
\item $C_{X_0}$, Hamiltonian chords and orbits located near $\{-R\}\times Y_0$. If $\gamma\in C_{X_0}$ then $\mathfrak{a}(\gamma)>0$.
\item $C_{X_0}^{-}$, Hamiltonian chords and orbits located near $\{0\}\times Y_0$. Given $\nu>0$, if $R$ is chosen small enough then every $\gamma\in C^-_{X_0}$ has $\mathfrak{a}(\gamma)<0$. 
\item $O_{X_{01}}$, the chords and orbits that correspond to critical points of $H_0$ in $X_{01}$. If $\gamma\in O_{X_{01}}$ then $\mathfrak{a}(\gamma)<0$.
\item $C_{X_1}^{-}$, Hamiltonian chords or orbits located near $\{0\}\times Y_1$. If $a_0<\nu(1-e^{-R})$ then for any chord or orbit $\gamma^{-}\in C_{X_1}^{-}$ we have $\mathfrak{a}(\gamma^{-})<0$.
\end{itemize}
\end{lemma}

\begin{proof}
Straightforward calculation.
\end{proof}

\subsection{Non-positive 1-forms of Hamiltonian vector fields}\label{sec:hamilt1form}
Let $H$ be a one step time independent Hamiltonian and $H_{t}$, $t\in[0,1]$ an associated time dependent one step Hamiltonian. 
We next define non-positive 1-forms with values in Hamiltonian vector fields parameterized by splitting compatible constant sections of $\mathcal{E}$. As before, our construction is inductive.

Denote $I=[0,1]$ and $S^1=I/\partial I$. For cylinders, strips, and cylinders with a slit with coordinates $s+it$, $s\in\R$, $t\in I$ or $t\in I/\partial I$ we use the time dependent Hamiltonian throughout and define
\[
B=X_{H_{t}}\otimes dt.
\]
For $x\in X$, the associated 2-form is $d(H_{t}(x)\,dt)=0$ and $B$ is non-positive.

Consider next disks and spheres in $D_{h;hm,k}$ with $hm+k=2$. Fix a cut-off function $\psi\colon D_{h;hm,k}\to [0,1]$ which equals $0$ outside the cylindrical ends, which equals $1$ in a neighborhood of each cylindrical end, and such that $d\psi$ has support in the regions of concavity only. Furthermore, we take the cut-off function to depend on the first coordinate only in the cylindrical end $[0,\infty)\times S^{1}$ or $(-\infty,0]\times S^1$ at interior punctures and $[0,\infty)\times[0,1]$ or $(-\infty,0]\times [0,1]$ at boundary punctures. Let $H_{t}$, $t\in I$, denote the time dependent one step Hamiltonian and $H$ the time independent one, chosen such that $H_{t}(x)\ge H(x)$ for all $(x,t)\in X\times I$. Let $H^{\psi}_{t}=(1-\psi)H+\psi H_{t}$. Define
\[
B=X_{H^{\psi}_{t}}\otimes\beta.
\]
For $x\in X$ the associated $2$-form is as follows: in the complements of cylindrical ends near the punctures it is given by
\[
d(H^{\psi}_{t}(x)\beta)=H(x)d\beta\le 0,
\]
and in the cylindrical ends near the punctures, with coordinates $s+it$, by
\[
d(H^{\psi}_{t}(x)\beta)=
\psi'(s)(H_{t}(x)-H(x))ds\wedge \beta+ ((1-\psi)H+\psi H_{t})d\beta\le 0, 
\]
where the last inequality holds provided $H_{t}$ is sufficiently close to $H$, so that the second term dominates when the first is non-vanishing. (Here we used that $dt\wedge\beta=0$ in the cylindrical end.) We now extend this field of 1-forms with values in Hamiltonian vector fields over all of $\mathcal{D}$ using induction. For one level curves in the interior of $\mathcal{D}_{h;hm,k}$ a straightforward extension of the above including more than two ends gives a non-positive form. For several level curves, gluing the 1-forms of the components define forms with desired properties in a neighborhood of the boundary of $\mathcal{D}_{h;hm,k}$. Finally, we interpolate between the two fields of forms over a collar region near the boundary using the interpolation of the form part $\beta$, see Lemma \ref{lemma:ell-levelsupharm}. We denote the resulting form with non-positive differential by $B$. 

Consider next the case of two step Hamiltonians. 
As for the one level Hamiltonians we insert a small time dependent perturbation near all Reeb orbits of positive action and we get a 0-order term $B$ exactly as above, just replace the one step Hamiltonian with the two step Hamiltonian everywhere.

We will consider one further type of 1-form with values in Hamiltonian vector fields that we use to interpolate between one step and two step Hamiltonians. Let $H_0=H$ be the two step Hamiltonian above and let $H_1$ be a one step Hamiltonian on $X_1$ with $H_{1}\le H_0$ everywhere.  Let $\phi_{T}\colon\R\to\R$ be a smooth function with non-positive derivative supported in $[-1,1]$ such that $\phi_{T}=1$  in $(-\infty,-1+T]$ and $\phi_{T}=0$ in $[1+T,\infty)$. Recall the superharmonic field of functions $g=g_e$, $e\in\mathcal{E}$ and let $B_{0}$ and $B_{1}$ be the fields of 1-forms parameterized by $\mathcal{E}$ associated to $H_0$ and $H_{1}$, respectively, constructed above. Fix a diffeomorphism $T\colon (0,1)\to\R$. Then the \emph{interpolation form}
\begin{equation}\label{eq:interpol1form}
B_{\tau}=(1-\phi_{T(\tau)}\circ g)B_{1}+(\phi_{T(\tau)}\circ g)B_{0} 
\end{equation}
is a 1-form with values in Hamiltonian vector fields (of the Hamiltonian $(1-\phi_{T(\tau)})H_1+\phi_{T(\tau)} H_0$). We check that it is non-positive. For fixed $x\in X_1$, the associated 2-form is 
\begin{align*} 
dB_{\tau}&=d((1-\phi_{T}\circ g)H_1(x)\beta+(\phi_{T}\circ g)H_{0}(x)\beta)\\
&=(1-\phi_{T}\circ g)d(H_{1}\beta)+(\phi_{T}\circ g)d(H_{0}\beta)\\
&+\phi_{T}'(g)(H_0(x)-H_1(x))dg\wedge\beta\le 0,
\end{align*}   
where the inequality follows since the first term is a convex combination of non-positive forms and the second is non-positive as well since $\beta=-i^{\ast}dg$. Note also that for $\tau=0$ and $\tau=1$, $B_{\tau}=B_{0}$ and $B_{\tau}=B_{1}$, respectively.

\subsection{Determinant bundles and orientations}
We use the field of 1-forms $B$ parameterized by constant splitting compatible sections of $\mathcal{E}$ and almost complex structures over $\mathcal{D}$ to define the Floer equation
\[ 
\bar{\pa}_{F}u=(du-B)^{0,1}=0
\]
for $u\colon (D_{h;hm,k},\partial D_{h;hm,k})\to (X,L)$. In order to study properties of the solution space we will consider the corresponding linearized operator $L\bar\pa_{F}$ which maps vector fields  $v$ with one derivative in $L^{p}$ into complex anti-linear maps $L\bar\pa_{F}(v)\colon T_{z}D_{h;hm,k}\to T_{u(z)} X$, in case of non-empty boundary the vector fields are tangent to $L$ along the boundary. The linearized operator is elliptic and it defines an index bundle over the space of maps. This index bundle is orientable provided the Lagrangian $L$ is relatively spin as was shown in \cite{FO3}. In this paper we will not use specifics of the index bundle beyond it being orientable. We will however use it to orient solution spaces of the Floer equation. For that purpose we fix capping operators for each Hamiltonian chord and orbit and use linear gluing results to find a system of coherent orientations of the index bundle. The main requirement here is that the positive and negative capping operator at each chord or orbit glues to the operator on a disk or sphere which has a fixed orientation of the index bundle over domains without punctures. The details of this linear analysis are similar to \cite{EES_ori} for chords and      
\cite{EGH} for orbits. There is however one point where the situation in this paper differs. Namely, our main equation depends on extra parameters corresponding to the simplex and the orientations we use depend on this. In order to get the right graded sign behavior for our Hamiltonian simplex DGA we will use the following conventions.

The index bundle corresponding to the parameterized problem is naturally identified with the index bundle for the un-parameterized problem stabilized by the tangent space of the simplex. Here we use the following orientation convention for the simplex. The simplex is given by the equation
\[ 
w_1+\dots+w_m = 1
\]  
and we think of its tangent space stably as the kernel-cokernel pair $(\R^{m},\R)$. We use the standard oriented basis $\pa_1,\dots,\pa_m$ of $\R^{m}$ and $\pa_0$ of $\R$. We then think of the direction $\pa_j$ as a stabilization of the capping operator of the $j^{\rm th}$ negative puncture and of $\pa_0$ as a stabilization of that at the positive puncture and get the induced orientation of the index bundle over $\mathcal{E}$ by gluing these stabilized operators. Then the index bundle orientations reflect Conley-Zehnder/Maslov grading in the DGA as usual. We give a more detailed discussion of index bundles and sign rules in the DGA in Appendix~\ref{A:1}.

\section{Properties of Floer solutions}
In this section we establish two basic results about Floer holomorphic curves. First we prove that the $\R$-factor of any Floer holomorphic curve in the cylindrical end of a Weinstein manifold satisfies a maximum principle. This result allows us to establish the correct form of Gromov-Floer compactness for our theories. Second we establish an elementary energy bound that ensures our Floer equations do not have any solutions with only high-energy asymptotes near the boundary of the parameterizing simplex. 

\subsection{A maximum principle for solutions of Floer equations}\label{s:maxprinciple}
Consider a $1$-parameter family of fields of splitting compatible $1$-forms $B=B_{\tau}$, $\tau\in[0,1]$, parameterized by constant sections of $\mathcal{E}$ and constructed from one step and two step Hamiltonians as in Section \ref{sec:hamilt1form}. (Fields of forms constructed from a one step Hamiltonian only, appear here as special cases corresponding to constant $\tau$, $\tau = 1$.) Let $J$ be a splitting compatible field of almost complex structures over $\mathcal{D}$. Recall that this means in particular that if $S=D_{h;hm,k}\in\mathcal{D}_{h;hm,k}$ then $J_{z}$ is an almost complex structure on $X$ for each $z\in S$ such that in any cylindrical end with coordinate $s+it$, $J_{s+it}=J_{t}$, , see Section \ref{sec:acs}.

We make the following \emph{non-degeneracy} assumption. The one and two step Hamiltonians are both linear at infinity $H(r,y)=h(r)=ar+b$ for real constants $a>0$ and $b$. We assume that the length $\ell$ of any Reeb orbit or Reeb chord satisfies
\begin{equation}\label{eq:non-degeneracy}
\ell\ne a.
\end{equation}  
Note that the set of Reeb chord and orbit lengths is discrete and hence the condition on the Hamiltonians holds generically. 

Consider now a solution $u\colon S\to X$ of the Floer equation
\[ 
(du-B)^{0,1}=0,
\]
where the complex anti-linear component of the map $(du-B)\colon T_{z} S\to T_{u(z)}X$ is taken with respect to the almost complex structures $J_{z}$ on $X$ and $j$ on $S$. 

\begin{lemma}\label{lem:belowpos}
If the non-degeneracy condition at punctures \eqref{eq:non-degeneracy} is satisfied then $u(S)$ is contained in the compact subset $\{r\le 1\}$. 
\end{lemma}

\begin{proof}
Assume that there exists $z\in S$ such that 
\begin{equation}\label{eq:wrong}  
r(z)=r(u(z))> 1.
\end{equation} 

Following \cite[\S7]{Abouzaid-Seidel}, we show that equation~\eqref{eq:wrong} leads to a contradiction. Fix a regular value $r'>1$ of the smooth function $r\circ u$ such that
\[ 
S'=\{z\in S\colon r\circ u(z)\ge r'\}\neq\varnothing.
\]
Then $S'$ is a Riemann surface with boundary with corners and its boundary can be decomposed as $\partial S'=\partial_{r'} S' \cup \partial_{L} S'$, where $u(\partial_{r'} S')\subset \{r=r'\}$ and $u(\partial_{L} S')\subset L$. Here both $\partial_{r'} S'$ and $\partial_{L} S'$ are finite unions of circles and closed intervals. The intervals in $\partial_{r'} S'$ and $\partial_{L} S'$ intersect at their endpoints that are the corners of $\partial S'$. 

Define the energy of $u\colon S\to X$ to be 
$$
E(u)=\frac 1 2\int_S \|du-B\|^2,
$$
where we measure the norm with respect to the metric $\omega(\cdot,J_{z}\cdot)$.
A straightforward computation shows that 
$$
E(u)=\int_S u^*\omega - u^*dH_{z}\wedge \beta.
$$

Recall the 2-form $\theta(x)=d(H_z(x)\beta)$ associated to $B$ for fixed $x\in X$,  and recall that the non-positivity of $B$ says that $\theta(x)$ is a non-positive 2-form for each $x\in X$. Consider the energy of $S'$, 
\begin{align*}
E(u|_{S'})  &=  \int_{S'} u^*\omega - u^*dH_{z}\wedge \beta \\
&\le \int_{S'} u^*\omega - u^*dH_{z}\wedge \beta - \theta(u(z))\\
&=  \int_{\partial S'} u^*\,r\alpha - H_{z}(u(z))\beta.
\end{align*} 
Since $\alpha|_L=0$ and  $\beta|_{\partial S}=0$, and since $H_{z}(r,y)=ar+b$ in the region $\{r\ge 1\}$ where $b<0$ for $a>0$ sufficiently large, the last integral satisfies 
\begin{align*}
\int_{\partial S'} u^*\,r\alpha - H_{z}(u(z))\beta &= \int_{\partial_{r'} S'} u^*\,r\alpha - a\,u^*r\,\beta-b\beta\\
&= \int_{\partial_{r'} S'} u^*\,r\alpha - a\,u^*r\,\beta -b\int_{S'}d\beta\\
&\le \int_{\partial_{r'} S'} u^*\,r\alpha - a\,u^*r\,\beta\\
&=r'\int_{\partial_{r'} S'} \alpha\circ (du-X_H\otimes \beta)\\
&=r'\int_{\partial_{r'} S'} \alpha\circ J_{z}\circ (du-X_H\otimes\beta)\circ (-i)\le 0,
\end{align*}
where $i$ is the complex structure on $S$. Here we use the identities $\alpha\circ J_{z}=dr$ and $dr(X_{H_{z}})=0$. The last inequality uses that $u(S')$ is contained in $\{r\ge r'\}$. Indeed, if $v$ is a positively oriented tangent vector to $\partial_{r'} S'$, then $-i v$ points outwards, and therefore $d(r\circ u)(-iv)\le 0$. 

We find that $E(u|_{S'})\le 0$, which implies that $u$ satisfies $du-X_H\otimes\beta=0$ on $S'$. Since $u$ intersects the level $r=r'$, it then follows that the image of any connected component of $S'$ under $u$ is contained in the image of a Reeb orbit or chord in this level set. Note that this conclusion is independent of the choice of regular level set $r'>1$ such that $S'=u^{-1}(\{r\ge r'\})\neq\varnothing$. Since such regular level sets exist (and are actually dense) in the interior of the original interval $(1,r')$, we get a contradiction. 
\end{proof}

\subsection{An action bound} \label{sec:action_bound}
In this section we establish an elementary action bound that we will use to show that our $\mathcal{E}$-families of Floer equations have no solutions with only high energy asymptotes near the boundary of the fiber simplex. Consider a Weinstein manifold with an exact Lagrangian submanifold $(X,L)$ and let $H_{t}$ be a one or two step time dependent Hamiltonian as above and let $\epsilon_{0}>0$ denote the smallest value of the action 
\begin{equation}\label{eq:action}
\mathfrak{a}(\gamma)=\int_{\gamma} \alpha - H_{t} dt
\end{equation}
of a Hamiltonian chord or orbit $\gamma$ corresponding to a Reeb chord or orbit. Then any high energy chord or orbit has action at least $\epsilon_0$. 

Let $u\colon S\to X$, $S=D_{h;hm,k}$ be a solution of the Floer equation 
\[ 
(du-B)^{0,1}=0
\]
asymptotic at the positive puncture to a periodic orbit or chord $\gamma$. 

\begin{lemma}\label{l:lowenergy} 
There are constants $L,\epsilon>0$ such that the following holds for any $L'\ge L$ and $0<\epsilon'\le \epsilon$. If there is a strip region $V=[0,L']\times I$ or cylinder region $V=[0,L']\times S^{1}$ in $S$ of length $L'$ that separates a negative puncture $q$ from the positive puncture $p$ and such that $B= X_{H}\otimes \epsilon dt$, in standard coordinates $s+it$ in $V$, where $H$ does not depend on $s+it$, then $q$ maps to a low energy chord or orbit.  
\end{lemma}   

\begin{remark}\label{r:consts}
As will be seen from the proof, the constant $L>0$ depends only on $\epsilon_0$, the action $E_{+}$ of the Hamiltonian chord or orbit at $p$, $M=\max_{\{r\le 1\}} H$, and $C=\max_{\{r\le 1\}} \|\alpha\|$, while the constant $\epsilon>0$ depends also on $F=\max_{\{r\le 1\}} \|X_H\|$.  
\end{remark}

\begin{proof} 
We use notation as in the proof of Lemma \ref{lem:belowpos} and Remark \ref{r:consts}.
Consider the energy
\begin{align*} 
E&=\int_{S}\|du-B\|^{2}=\int_{S} u^{\ast}\omega - u^*dH_{z}\wedge\beta\\
&\le \int_{S} u^{\ast}\omega - u^*dH_{z}\wedge\beta -\theta(u(z))\\
&= E_+ - \sum_{i=1}^{hm+k} E_{i,-},
\end{align*}
where $E_+$ is the action at the positive puncture and $E_{i,-}$, $i=1,\dots,hm+k$ are the actions at the negative punctures.
In particular, the action $E_{i,-}$ at any of the negative punctures satisfies
\begin{equation}\label{eq:negative-estimate}
E_{i,-}\le E_+.
\end{equation}
Also, because each of the actions $E_{i,-}$ is positive, we have 
$$
E\le E_+.
$$

Consider now the contribution to the energy from the strip or cylinder region $V$. Fix $\eta>0$ and note that, in the strip case, the measure of the set of points $s\in[0,L]$ such that 
\[ 
\int_{\{s\}\times[0,1]}  \left\|\partial_t u- \epsilon X_{H}\right\|^{2}\,dt\ge \eta
\]
is bounded by $E/\eta$ (similarly for the integral over $\{s\}\times S^1$ in the cylinder case). 
In particular if $L>E/\eta$ we have that there are slices $\gamma'=\{s_0\}\times[0,1]$ in the strip case or $\gamma'=\{s_0\}\times S^{1}$ in the cylinder case for which $\left\|\partial_t u(s_0,\cdot)- \epsilon X_{H}\right\|_{L^2}^{2}\le \eta$, which implies 
$$
\|\partial_t u(s_0,\cdot)\|_{L^2}\le \sqrt \eta + \epsilon F.
$$
We obtain for the action of $\gamma'$ the estimate
\begin{eqnarray*}
\left|\int_{\gamma'} \alpha - \epsilon H dt\right| & \le & C\|\p_t u(s_0,\cdot)\|_{L^1} + \epsilon M \\
& \le & C\sqrt \eta + \epsilon (CF + M).
\end{eqnarray*}
Applying Stokes' theorem to the energy integral of the part $S'$ of $S$ containing $\gamma'$ and the negative puncture $q$ then shows as in~\eqref{eq:negative-estimate} that the energy of the chord or orbit at $q$ is $<\epsilon_0$, provided $\eta=E/L$ and $\epsilon$ are sufficiently small. 
\end{proof}

\section{Properties of spaces of Floer solutions}\label{S:mdlispaces}
Let $(X,L)$ be a Weinstein pair as before. 
Consider a $\mathcal{E}$-family $B_{\tau}$, $\tau\in[0,1]$ of interpolation splitting compatible 1-forms over $\mathcal{D}$ with values in Hamiltonian vector fields, see Section \ref{sec:hamilt1form}, and a field of domain dependent almost complex structures, see Section \ref{sec:acs}, where the Hamiltonians satisfy the non-degeneracy condition at infinity. Here we think of $(X,L)=(X_1,L_1)$ constructed from a cobordism if $\tau\in[0,1)$, and if $\tau=1$ we also allow standard Weinstein pairs.

This data allows us to study the Floer equation
\begin{equation}\label{eq:Floerspec} 
(du-B_{\tau})^{0,1}=0,
\end{equation}
for $u\colon (D_{h;hm,k},\partial D_{h;hm,k})\to (X,L)$. We will refer to solutions of \eqref{eq:Floerspec} as \emph{Floer holomorphic curves}.

\subsection{Transversality and dimension}\label{sec:Floermodulispaces}
In order to express the dimensions of moduli spaces of Floer holomorphic curves, we use Conley-Zehnder indices for chords and orbits (with conventions as in \cite[Appendix~A.1]{CEL}). They are defined as follows. If $\gamma$ is a Hamiltonian orbit then fix a disk $D_{\gamma}$ (recall that we assume $\pi_1(X)=1$) that bounds $\gamma$ and a trivialization of the tangent bundle $TX$ over $D_\gamma$. The Conley-Zehnder index $\CZ(\gamma)\in\mathbb{Z}$ of a Hamiltonian orbit is then defined using the path of linear symplectic matrices that arises as the linearization of the Hamiltonian flow along $\gamma$ in this trivialization, see~\cite{RS}. Then $\CZ(\gamma)$ is independent of the choice of trivialization since $c_{1}(X)=0$. 

If $c$ is a Hamiltonian chord we pick a capping disk $D_{c}$ mapping the unit disk into $X$ as follows. Pick a base point in each component of the Lagrangian $L$. Fix paths connecting base points in different components and along these paths fix paths of Lagrangian tangent planes connecting the tangent planes of the Lagrangian $L$ at the base points. (We use the constant path with the constant tangent plane at the base point connecting the base point in a given component to itself.) In the disk $D_c$ we map the boundary arc $\pa D_{c}^{-}$ between $-1$ and $1$ to the Hamiltonian chord, and we map the boundary arc $\pa D_{c}^{+}$ between $1$ and $-1$ as follows: the boundary arc between $1$ and $e^{\frac{\pi i}{4}}$ is mapped to the component of $L$ that contains the Hamiltonian chord endpoint, and connects the latter to the base point; the arc between $e^{\frac{\pi i}{4}}$ and $e^{\frac{3\pi i}{4}}$ follows the path between base points; finally, the arc between $e^{\frac{3\pi i}{4}}$ and $-1$ is mapped to the connected component of $L$ that contains the Hamiltonian chord start point, and connects the base point to the Hamiltonian chord start point. This then gives the following loop $\Gamma_{c}$ of Lagrangian planes: along $\pa D_{c}^{+}$ we follow first the tangent planes of $L$ starting at the endpoint of the chord and ending at the base point, then the planes along the path connecting base points, then again planes tangent to $L$ from the base point to the start point of the chord; along $\pa D_{c}^{-}$ we transport the tangent plane of the Lagrangian at the chord start point by the linearization of the Hamiltonian flow along the chord, and finally we close up by a rotation along the complex angle in the positive direction connecting the transported Lagrangian plane to the tangent plane at the endpoint of the chord. We define
\[
\CZ(c)=\mu(\Gamma_c),
\]
where $\mu$ denotes the Maslov index of $\Gamma_c$ read in a trivialization of $TX$ over $\pa D_{c}$ that extends over $D_{c}$. This is then well-defined since $c_1(X)=0$ and since the Maslov class of $L$ vanishes. 
\begin{remark}
The Conley-Zehnder index $\CZ(c)$ of a Reeb chord $c$ with both endpoints in one component of the Lagrangian submanifold is independent of all choices. For chords with endpoints in distinct components $\CZ$ is independent up to an over all shift that depends on the choice of tangent planes along the path connecting base points.  
\end{remark} 

We also define positive and negative capping operators. For chords $c$ these operators $o_\pm(c)$ are defined using capping disks. This capping operator is a linearized Floer-operator on a once boundary-punctured disk, with Lagrangian boundary condition given by the tangent planes along the capping path oriented from the endpoint of the chord to the start point for the positive capping operator $o_+(c)$ and with the reverse path for the negative capping operator $o_-(c)$. We assume (as is true for generic data) that the image of the Lagrangian tangent plane at the start point of the chord under the linearized flow is transverse to the tangent plane at the endpoint. For orbits, the capping operators $o_{\pm}(\gamma)$ are operators on punctured spheres with positive or negative puncture with asymptotic behavior determined by the linearized Hamiltonian flow along the orbit $\gamma$. More precisely, the capping operators are then $\bar\partial$-operators perturbed by a 0-order term acting on the Sobolev space of vector fields $v$ on the punctured sphere $S$ or disk $D$ that in the latter case are tangent to the Lagrangian along $\partial D$ with one derivative in $L^{p}$, $p>2$.

We find that the chord capping operators are Fredholm and their index is given by the formula~\cite{MS04,EES}
\[
\ind(o_+(c))= n+(\CZ(c)-n)=\CZ(c),\quad \ind(o_-(c))=n-\CZ(c).
\]
The orbit capping operators have index~\cite{Abouzaid-cotangent,CEL}
\[
\ind(o_{+}(\gamma))=n+\CZ(\gamma),\quad \ind(o_{-}(\gamma))=n-\CZ(\gamma).
\]

Let $a$ be a Hamiltonian chord, $\gamma$ a Hamiltonian orbit, $\mathbf{b}=b_1\dots b_m$ a word of Hamiltonian chords and $\bo{\eta}=\eta_1\dots\eta_k$ a word of Hamiltonian orbits. Let $\sigma\in\mathcal{E}$ be a splitting compatible constant section over $\mathcal{D}$, which takes values in the simplex $\Delta^{hm+k-1}$ over the interior of $\mathcal{D}_{h;hm,k}$. 

When the number of boundary components of the source curve is $h=1$, we consider the moduli space $\mathcal{F}^{\sigma}_{\tau}(a;\mathbf{b},\bo{\eta})$ of solutions 
\[
u\colon (D_{1;m,k},\partial D_{1;m,k})\to(X,L), \qquad D_{1;m,k}\in\mathcal{D}_{1;m,k}
\]
of the Floer equation 
\[ 
\left(du - B^{\sigma}_{\tau}\right)^{0,1}=0.
\]  
Here $B^{\sigma}_{\tau}$ is the 1-form with values in Hamiltonian vector fields determined by $\sigma\in\mathcal{E}$. The map $u$ converges at the positive puncture to $a$, and at the negative punctures to $b_1,\dots,b_m,\eta_1,\dots,\eta_k$. The interior negative punctures are endowed with asymptotic markers induced from the positive boundary puncture as in Section~\ref{sec:cylends}. We write
\[ 
\mathcal{F}_{\tau}(a;\mathbf{b},\bo{\eta})=\bigcup_{\sigma\in\Delta^{hm+k-1}}\mathcal{F}^{\sigma}_{\tau}(a;\mathbf{b},\bo{\eta}).
\]
(Recall that the family $B^{\sigma}$ depends smoothly on $\sigma$.) We also write
\begin{align*} 
\mathcal{F}_{\R}(a;\mathbf{b},\bo{\eta}) &=\bigcup_{\tau\in(0,1)}\mathcal{F}_{\tau}(a;\mathbf{b},\bo{\eta}),\\
\mathcal{F}_{\R}^{\sigma}(a;\mathbf{b},\bo{\eta}) &=\bigcup_{\tau\in(0,1)}\mathcal{F}_{\tau}^{\sigma}(a;\mathbf{b},\bo{\eta}).
\end{align*}

When the number of boundary components of the source curve is $h=0$, we similarly consider the moduli space $\mathcal{F}^{\sigma}_{\tau}(\gamma;\bo{\eta})$ of solutions  
\[
u\colon D_{0;0,k}\to X, \qquad D_{0;0,k}\in \mathcal{D}_{0;0,k}
\]
of the Floer equation 
\[ 
\left(du - B^{\sigma}_{\tau}\right)^{0,1}=0,
\]   
converging at the positive puncture to $\gamma$, and at the negative punctures to $\eta_1,\dots,\eta_k$. Here the positive puncture has a varying asymptotic marker, which induces asymptotic markers at all the negative punctures as described in Section~\ref{sec:cylends}. We write
\[ 
\mathcal{F}_{\tau}(\gamma;\bo{\eta})=\bigcup_{\sigma\in{\Delta}^{k-1}}\mathcal{F}^{\sigma}_{\tau}(\gamma;\bo{\eta})
\]
and 
\begin{align*} 
\mathcal{F}_{\R}(\gamma;\bo{\eta})&=\bigcup_{\tau\in(0,1)}\mathcal{F}_{\tau}(\gamma;\bo{\eta}),\\
\mathcal{F}_{\R}^{\sigma}(\gamma;\bo{\eta})&=\bigcup_{\tau\in(0,1)}\mathcal{F}_{\tau}^{\sigma}(\gamma;\bo{\eta}).
\end{align*}

\begin{remark}
Floer equations corresponding to one step Hamiltonians are a special case of the above, corresponding to $\tau=1$. We sometimes use a simpler notation for such spaces: we drop the $\tau=1$ subscript and write $\mathcal{F}^{\sigma}=\mathcal{F}_{1}^{\sigma}$ and $\mathcal{F}=\mathcal{F}_{1}$.
\end{remark}

\begin{theorem}\label{thm:dim+tv}
For generic families of almost complex structures and Hamiltonians, the moduli spaces $\mathcal{F}_{\tau}(\gamma,\bo{\eta})$, $\mathcal{F}_{\tau}(a;\mathbf{b},\bo{\eta})$,   
$\mathcal{F}_{\R}(\gamma,\bo{\eta})$, $\mathcal{F}_{\R}(a;\mathbf{b},\bo{\eta})$
are manifolds of dimension
\begin{align*}
\dim \mathcal{F}_{\tau}(\gamma;\bo{\eta})&=\dim \mathcal{F}_{\R}(\gamma;\bo{\eta})-1\\
&=(\CZ(\gamma)+(n-3))-\sum_{j=1}^{k}(\CZ(\eta_j)+(n-3))-1,
\end{align*}
and
\begin{align*}
\dim \mathcal{F}_{\tau}(a;\mathbf{b},\bo{\eta})&=\dim \mathcal{F}_{\R}(a;\mathbf{b},\bo{\eta})-1\\
&=(\CZ(a)-2)-\sum_{j=1}^{m}(\CZ(b_j)-2)\\
&-\sum_{j=1}^{k}(\CZ(\eta_j)+(n-3))-1,
\end{align*}
respectively.

For generic fixed $\sigma\in\Delta^{hm+k-1}$ the corresponding moduli spaces $\mathcal{F}^{\sigma}(\gamma,\bo{\eta})$ and $\mathcal{F}^{\sigma}(a;\mathbf{b},\bo{\eta})$ are manifolds of dimension
\begin{align*}
\dim \mathcal{F}_{\tau}^{\sigma}(\gamma;\bo{\eta})&=\dim \mathcal{F}_{\R}^{\sigma}(\gamma;\bo{\eta})-1\\
&=(\CZ(\gamma)+(n-2))-\sum_{j=1}^{k}(\CZ(\eta_j)+(n-2))-1,
\end{align*}
and
\begin{align*}
\dim \mathcal{F}^{\sigma}_{\tau}(a;\mathbf{b},\bo{\eta})&=\dim \mathcal{F}_{\R}^{\sigma}(a;\mathbf{b},\bo{\eta})-1\\
&=(\CZ(a)-1)-\sum_{j=1}^{m}(\CZ(b_j)-1)\\
&-\sum_{j=1}^{k}(\CZ(\eta_j)+(n-2))-1,
\end{align*}
respectively.

Furthermore, for generic data, the projection of the moduli spaces $\mathcal{F}_{\R}$ and $\mathcal{F}_{\R}^{\sigma}$ to the line $\R$ (interpolating between the Hamiltonians) is a Morse function with distinct critical values. 
\end{theorem}

\begin{proof}
To see this we first note that the operator we study is Fredholm. The expected dimension of the moduli space is then given by the sum of the index of the operator acting on a fixed surface and the dimension of auxiliary parameter spaces (i.e.~the space of conformal structures on the domain and the space which parameterizes the choice of 1-forms). 

Consider first the case when $h=1$. We denote the index of the operator on the fixed surface $\ind(a;\mathbf{b},\bo{\eta})$. To compute it, we glue on capping operators at all punctures. Additivity of the index under gluing at a non-degenerate chord or orbit together with the Riemann-Roch formula then gives:
\[
n=\ind(a;\mathbf{b},\bo{\eta})+n-\CZ(a) + \sum_{j=1}^{m}\CZ(b_j)+\sum_{j=1}^{k}(\CZ(\eta_j)+n).
\]
The dimension is then obtained by adding the dimension of the space of conformal structures and that of the space of 1-forms:
\[
\dim \mathcal{F}(a;\mathbf{b},\bo{\eta})= \ind(a;\mathbf{b},\bo{\eta})+(m-2)+2k+(m+k-1).
\] 
When the form $B^{\sigma}$ is fixed, we simply subtract the dimension of the simplex, $(m+k-1)$. The calculation in the case $h=m=0$ is similar and gives
\[
\dim \mathcal{F}(\gamma;\bo{\eta})= (\CZ(\gamma)+n)-\sum_{j=1}^{k}(\CZ(\eta_j)+n)+2k-3+(k-1),
\]
where $2k-3$ is the dimension of the space of conformal structures on the sphere with $k+1$ punctures where there is a varying asymptotic marker at one of the punctures. In the case where the form $B^{\sigma}$ is fixed we subtract the dimension of the simplex, $k-1$.

Finally, to see that these are manifolds, we need to establish surjectivity of the linearized operator for generic data. This is well-known in the current setup and follows from the unique continuation property of pseudo-holomorphic curves in combination with an application of the Sard-Smale theorem. The key points are that $J$ (and $H$) are allowed to depend on all parameters and that $(X,L)$ is exact so that no bubbling of pseudo-holomorphic spheres or disks occurs, see e.g. \cite[Appendix]{Bourgeois-pik} and~\cite[Section 9.2]{MS04}. 

The last statement is a straightforward consequence of the Sard-Smale theorem. 
\end{proof}

We next show that there are no solutions of the Floer equation with only high-energy asymptotes if the 0-order term corresponds to a constant section of $\mathcal{E}$ that lies sufficiently close the boundary of the simplex.

\begin{lemma}\label{l:nohigh}
For any $E>0$ there exists $\epsilon>0$ such that if $\sigma$ is a constant section of $\mathcal{E}$ that lies in an $\epsilon$ neighborhood of the boundary of the simplex and if $\mathfrak{a}(a)<E$ and $\mathfrak{a}(\gamma)<E$, then for any non-trivial words $\mathbf{b}$ and $\bo{\eta}$ of high-energy chords and orbits, respectively, and for any $\tau\in[0,1]$, the moduli spaces $\mathcal{F}^{\sigma}_{\tau}(a;\mathbf{b},\bo{\eta})$ and $\mathcal{F}^{\sigma}_{\tau}(\gamma;\bo{\eta})$ are empty. 
\end{lemma}

\begin{proof}
This is an immediate consequence of the $\ell$-level condition on the non-negative 1-form $\beta$ and Lemma~\ref{l:lowenergy}.
\end{proof}

\subsection{Compactness and gluing} \label{sec:compactness-gluing}
For simpler notation, we write $\mathcal{F}_{\tau}$, $\mathcal{F}^{\sigma}_{\tau}$, $\mathcal{F}_{\R}$, and $\mathcal{F}^{\sigma}_{\R}$ with unspecified punctures as common notation for either type of moduli space (corresponding to either $h=0$ or $h=1$) in Theorem \ref{thm:dim+tv}. We also write $\mathcal{F}^{+}_{\tau}$ and $\mathcal{F}^{+}_{\R}$ for components of $\mathcal{F}_{\tau}$ and $\mathcal{F}_{\R}$ where all asymptotic chords and orbits are of high-energy. Recall that, if $B^{\sigma}$ is a splitting compatible field of 1-forms determined by a constant splitting compatible section $\sigma$ of $\mathcal{E}$ then, over a several level curve, $B^{\sigma}$ determines 1-forms depending on constant sections over its pieces. 

\begin{theorem} \label{thm:mdli}
The spaces $\mathcal{F}^{\sigma}_{\tau}$ and $\mathcal{F}^{+}_{\tau}$ admit compactifications as manifolds with boundary with corners, where the boundary corresponds to several level curves in $\mathcal{F}^{\sigma}_{\tau}$ and $\mathcal{F}^{+}_{\tau}$ respectively, joined at Hamiltonian chords or orbits. 
\end{theorem}

\begin{proof}
The fact that any sequence of curves in $\mathcal{F}^{\sigma}_{\tau}$ has a subsequence that converges to a several level curve is a well-known form of Gromov-Floer compactness for $(X,L)$ exact. In order to find a neighborhood of the several level curves in the boundary of the moduli space we use Floer gluing. That the Floer equation is compatible with degeneration in the moduli space of curves is a consequence of the gluing compatibility condition for the family of 1-forms $B^{\sigma}$, $\sigma\in\mathcal{E}$. Both compactness and gluing are treated in~\cite[Chapters~4 and~10]{MS04} and in~\cite[Chapter~9]{Seidel-book}, see also e.g.~\cite[Appendix A]{ESm} for a treatment of family gluing. 

For $\mathcal{F}^{+}_{\tau}$, by Lemma~\ref{l:nohigh} note that there are no solutions near the boundary of the simplex so the only possible boundary are broken curves joined at high-energy chords or orbits.
\end{proof}

We next consider compactifications of moduli spaces $\mathcal{F}_{\R}^{+}$ which consist of solutions of the Floer equation with the interpolation form $B_{\tau}$ as $\tau$ varies over $(0,1)$. Similar results hold for moduli spaces $\mathcal{F}_{\R}^{\sigma}$, but we focus on the high-energy case since that is all we use later and since we then need not involve any low-energy chords and orbits.
  
\begin{theorem} \label{thm:mdli2}
The moduli spaces $\mathcal{F}_{\R}^{+}$ admit compactifications as manifolds with boundary with corners, where the boundary corresponds to several level curves joined at Hamiltonian chords and orbits of the following form:
\begin{itemize}
\item Exactly one level $S$ (possibly of several components) lies in $\mathcal{F}_{\R}^{+}$.
\item At the positive punctures of $S$ are attached several level curves in $\mathcal{F}_1^{+}$ in $X_1$ that solve the Floer equation
\[ 
(du-B_1)^{0,1}=0.
\]
\item At the negative punctures of $S$ are attached several level curves in $\mathcal{F}_0^{+}$ in $X_1$
corresponding to the Floer equation
\[ 
(du - B_0)^{0,1}=0.
\]
In fact, the curves in $\mathcal{F}_{0}^{+}$ in $X_1$ can be canonically identified with the curves in $\mathcal{F}^{+}$ in $X_0$ that solve the Floer equation with $B_0$ constructed from the Hamiltonian that equals $H_0$ on $X_0$ that continues to grow linearly over the end of $X_0$.
\end{itemize}
\end{theorem}

\begin{proof}
The proof is a repetition of the proof of Theorem \ref{thm:mdli}, except for the last statement. The last statement follows from Lemma \ref{lem:belowpos} which shows that a curve with positive puncture at a chord or orbit in $C_{X_0}$ (notation as in Lemma \ref{lemma:actiondecomposition}) lies inside $\{r\le 1\}$, where $r=e^t$ is the coordinate in the symplectization end of $X_0$.
\end{proof}

\section{Definition of the Hamiltonian simplex DGA}\label{S:opendga}
In this section we define the Hamiltonian simplex DGA. In order to simplify grading and dimension questions we assume that $\pi_1(X)=0$, $c_1(X)=0$ and that the Maslov class $\mu_L$ of the Lagrangian submanifold $L$ vanishes, see Section~\ref{sec:examples} for a discussion of the general case.

\subsection{DGA for fixed Hamiltonian}\label{s:fixHamDGA}
Let $H$ be a one step time dependent Hamiltonian, see Section \ref{sec:hamilt1form}, and let $B$ be an $\mathcal{E}$-family of 1-forms associated to $H$ and fix a family of almost complex structures.
 
Define the algebra $\mathcal{SC}^+(X,L;H)$ to be the algebra generated by high-energy Hamiltonian chords $c$ of $H$, graded by 
$$
|c|=\CZ(c)-2,
$$ 
and by high-energy $1$-periodic orbits $\gamma$ of $H$ graded by 
$$
|\gamma|=\CZ(\gamma)+(n-3).
$$ 
We impose the condition that orbits sign commute with chords and that orbits sign commute with orbits. See also Remark~\ref{rmk:sign_commute}.

Define the map
\[ 
\delta\colon \mathcal{SC}^+(X,L;H)\to\mathcal{SC}^+(X,L;H),\qquad 
\delta=\delta_1+\delta_2+\dots+\delta_m+\dots,
\]
to satisfy the graded Leibniz rule and as follows on generators. For a Hamiltonian chord $a$:
\[ 
\delta_r(a)=\sum_{|a|-|\mathbf{b}|-|\bo{\eta}|=1}\frac{1}{k!}\left| \mathcal{F}(a;\mathbf{b},\bo{\eta})\right| \bo{\eta}\mathbf{b},
\]
where the sum ranges over all words $\mathbf{b}=b_1\dots b_m$ and $\bo{\eta}=\eta_1\dots\eta_k$ which satisfy the grading condition and are such that $m+k=r$. Here $\left|\mathcal{F}\right| $ denotes the algebraic number of elements in the oriented $0$-dimensional manifold $\mathcal{F}$. Similarly, for a Hamiltonian orbit $\gamma$:
\[ 
\delta_r(\gamma)=\sum_{|\gamma|-|\bo{\eta}|=1}\frac{1}{r!}\left| \mathcal{F}(\gamma;\bo{\eta})\right| \bo{\eta},
\]
where the sum ranges over all words $\bo{\eta}=\eta_1\dots\eta_r$ which satisfy the grading condition.  

\begin{lma} \label{lma:secondarydelta2}
The map $\delta\colon\mathcal{SC}^{+}(X,L;H)\to\mathcal{SC}^{+}(X,L;H)$ has degree $-1$ and is a differential, i.e.~$\delta\circ\delta=0$.
\end{lma}

\begin{proof}
This is a consequence of Theorem \ref{thm:mdli}: the terms in $\delta\circ \delta$ are in bijective sign preserving correspondence with the boundary components of the (oriented) 1-dimensional compactified moduli spaces $\mathcal{F}$. 
\end{proof}

\begin{remark}
Repeating the above constructions replacing the moduli spaces $\mathcal{F}$ with $\mathcal{F}^{\sigma}$ for some generic constant splitting compatible section of $\mathcal{E}$, we get a differential on the DGA $\mathcal{SC}^{+}(X,L;H)$ with grading shifted up by $1$, denoted $\mathcal{SC}^{+}(X,L;H)[-1]$.  
\end{remark}

\subsection{Cobordism maps for fixed Hamiltonians}\label{s:fixHamcob}
Consider a symplectic cobordism $(X_{10},L_{10})$ and fix a two step Hamiltonian $H_0$ and a one step Hamiltonian $H_1$. As in Theorem \ref{thm:mdli2} we think of $H_0$ also as a Hamiltonian on $X_0$ only (basically removing the second step making it a one step Hamiltonian).

Define the map 
\begin{equation}\label{eq:preldefcobmap} 
\Phi\colon \mathcal{SC}^{+}(X_1,L_1;H_1)\to\mathcal{SC}^{+}(X_0,L_0;H_0), \quad \Phi=\Phi_1+\Phi_2+\dots
\end{equation}
as the algebra map given by the following count on generators. 
\begin{itemize}
\item For chords $a$:
\[ 
\Phi_r(a)=\sum_{|a|-|\mathbf{b}|-|\bo{\eta}|=0} \frac{1}{k!}\left| \mathcal{F}_{\R}(a;\mathbf{b},\bo{\eta})\right| \bo{\eta}\mathbf{b},
\]
where the sum ranges over all $\mathbf{b}=b_1\dots b_m$ and $\bo{\eta}=\eta_1\dots\eta_k$ with $m+k=r$.
\item For orbits $\gamma$:
\[ 
\Phi_{r}(\gamma)=\sum_{|\gamma|-|\bo{\eta}|=0}\frac{1}{r!}\left| \mathcal{F}_{\R}(\gamma;\bo{\eta})\right| \bo{\eta}.
\]
\end{itemize}

\begin{theorem}\label{thm:chainmap}
The map $\Phi\colon \mathcal{SC}^{+}(X_1,L_1;H_1)\to\mathcal{SC}^{+}(X_0,L_0;H_0)$ is a chain map, i.e.~$\delta\Phi=\Phi \delta$.
\end{theorem}

\begin{proof}
By Theorems \ref{thm:dim+tv} and \ref{thm:mdli2},
contributions to $\Phi \delta-\delta\Phi$ correspond to the boundary of an oriented 1-dimensional moduli space.   
\end{proof}

\subsection{The Hamiltonian simplex DGA}
In order for the DGA introduced above to capture all aspects of the Reeb dynamics at the boundary of the Weinstein pair $(X,L)$ we need to successively increase the slope of the Hamiltonian. Consider a family of one level Hamiltonians $H_a$ where $H_{a_1}>H_{a_0}$ if $a_1>a_0$ and $H_a(r,y)=ar+b$ in the cylindrical end $[0,\infty)\times Y$. Inserting a trivial cobordism, we change $H_{a_1}$ to a two step Hamiltonian $H_{a_1}'$ with the slope $a_1$ at the end of the trivial cobordism as well. Then $H_{a_1}'>H_{a_0}$. We define the \emph{Hamiltonian simplex DGA}
$$
\mathcal{SC}^{+}(X,L)=\underrightarrow{\lim}_{\,a\to\infty}\,\mathcal{SC}^{+}(X,L;H_{a}),  
$$
where the direct limit is taken with respect to the directed system given by the cobordism maps
\[ 
\Phi\colon \mathcal{SC}^{+}(X,L;H_{a_0})\to \mathcal{SC}^{+}(X,L;H_{a_1}')= \mathcal{SC}^{+}(X,L;H_{a_1}),
\]
see Sections~\ref{sec:compactness-gluing} and~\ref{s:hmtpyofcobmaps} for the last equality. Its homology is 
$$
\mathcal{SH}^{+}(X,L)=\underrightarrow{\lim}_{\,a\to\infty}\, H(\mathcal{SC}^{+}(X,L;H_a)).
$$

\begin{remark}
One can alternatively define the Hamiltonian simplex DGA $\mathcal{SC}^+(X,L)$ as the homotopy limit of the directed system $\{\mathcal{SC}^+(X,L;H_{a})\}$, obtained by the algebraic mapping telescope construction as in~\cite[\S3g]{Abouzaid-Seidel} (see also~\cite[Chapter~3, p.~312]{Hatcher}).
\end{remark}

\subsection{Homotopies of cobordism maps}\label{s:hmtpyofcobmaps}
In this subsection we study invariance properties of the cobordism maps defined in Section~\ref{s:fixHamcob}. As a consequence we find that the homotopy type of the Hamiltonian simplex DGA is independent of Hamiltonian, 0-order perturbation term, and field of almost complex structure and depends only on the underlying Weinstein pair $(X,L)$.

Let $(X_{10},L_{10})$ be a cobordism of pairs and consider a 1-parameter deformation of the data used to define the cobordism map parameterized by $s\in I$. We denote the corresponding cobordism maps
\[ 
\Phi_{s}\colon \mathcal{SC}^{+}(X_1,L_1)\to \mathcal{SC}^{+}(X_0,L_0), \quad s\in I.
\]

Here we take the deformation of the data to be supported in the middle region of the cobordism. In other words the symplectic form, the field of almost complex structures, and the Hamiltonians and associated 0-order terms in the Floer equation vary in the compact cobordism but are left unchanged near $[-R,0]\times Y_{0}$ and outside $\{0\}\times Y_1$, see Figure \ref{fig:H-cob}. 

For fixed $s\in I$ we get an interpolation form $B_{\tau}^{s}$, $\tau\in I$ and moduli spaces $\mathcal{F}_{\R}^{s}$, as in Section~\ref{sec:compactness-gluing}. Exactly as there, we suppress from the notation the punctures, and also the constant section $\sigma$ on which $B_\tau^s$ depends. We write the corresponding parameterized moduli spaces as
\[ 
\mathcal{F}_{\R}^{I}=\bigcup_{s\in I}\mathcal{F}_{\R}^{s}.
\]

We will show below that the chain maps $\Phi_{0}$ and $\Phi_{1}$ are chain homotopic. The proof is however rather involved. To explain why we start with a general discussion pointing out the main obstruction to a simple proof.
The chain maps $\Phi_{0}$ and $\Phi_{1}$ are defined by counting $(-1)$-disks in $\R$-families of Floer equations, or in other words rigid $0$-dimensional curves in $\mathcal{F}^{0}_{\R}$ and $\mathcal{F}^{1}_{\R}$, respectively. A standard transversality argument shows that for generic $1$-parameter families $s\in I$, the $0$-dimensional components of the moduli spaces $\mathcal{F}^{I}_{\R}$ constitute a transversely cut out oriented 0-manifold. From the point of view of parameterized Floer equations this 0-manifold consists of isolated $(-2)$-disks, where one parameter is $\tau\in I$ and the other is $s\in I$.   

\begin{remark}\label{rmk:-2curve}
In our notation below we always include the simplex parameters in the dimension counts but view both the interpolation parameter $\tau\in I$ and the deformation parameter $s\in I$ as extra parameters. With this convention we call a curve of formal dimension $d$ a \emph{$(d)$-curve}.
\end{remark}

In analogy with the definition of the chain maps induced by cobordisms, counting $(-2)$-curves during a generic deformation of cobordism data should give a chain homotopy between the chain maps $\Phi_{0}$ and $\Phi_{1}$ at the ends of the deformation interval $I$. However, counting $(-2)$-curves is not entirely straightforward in the present setup because of the following transversality problem: since the curves considered may have several negative punctures mapping to the same Hamiltonian chord or orbit, an isolated $(-2)$-curve can be glued to the negative ends of a $(d)$-curve, $d>0$ a number of $(d+1)$ times, resulting in a several level curve of formal dimension
\[ 
d+(d+1)(-2+1)=-1,
\]
and could thus contribute to the space of $(-1)$-curves used to define the cobordism chain maps. In order to find a chain homotopy the $(-2)$ curve should appear only once in combination with the $(0)$-curve that gives the differentials. 

To resolve this problem, we restrict attention to a small time interval around the critical $(-2)$-curve moment and ``time-order'' the negative ends of the curves in the moduli space of Floer holomorphic curves in the positive end. Similar arguments are used in e.g.~\cite{Ekholm_rsft, Abouzaid-Seidel}. In these constructions there are differences between interior and boundary punctures. In case the positive puncture of the $(-2)$-disk is a chord (boundary puncture) the time ordering argument is simpler since there is a natural order of the boundary punctures in the disks where the $(-2)$-disk can be attached, and that ordering can be used in building the perturbation scheme. In the orbit case (interior puncture) there is no natural ordering and we are forced to add a homotopy of homotopies argument on top of the ordering perturbation. We sketch these constructions below but point out that actual details do depend on the existence of a suitable perturbation scheme that will not be discussed here, see Remark~\ref{rmk:perturbations}.      

We now turn to the proof that $\Phi_0$ and $\Phi_1$ are chain homotopic. Consider first the case in which there are no $(-2)$-curve instances in the interval $[0,1]$. Then the 1-dimensional component of $\mathcal{F}^{I}_{\R}$ gives an oriented cobordism between the 0-dimensional moduli spaces used to define the cobordism maps and hence $\Phi_0=\Phi_1$. A general deformation can be perturbed slightly into general position and then it contains only a finite number of transverse $(-2)$-curve instances. By subdividing the family it is then sufficient to show that $\Phi_0$ and $\Phi_1$ are homotopic for deformation intervals that contain exactly one such transverse $(-2)$-curve. The following result expresses the effect of a $(-2)$-disk algebraically. The proof is rather involved and occupies the rest of this section.

\begin{lemma}\label{lem:chainhomotopy_orbit}
Assume that the deformation interval contains exactly one $(-2)$-curve. Then
the DGA maps $\Phi_{0}$ and $\Phi_{1}$ are chain homotopic, i.e.~there exists a degree $+1$ map $K\colon \mathcal{SC}^{+}(X_1,L_1)\to \mathcal{SC}^{+}(X_0,L_0)$, such that
\begin{equation}\label{eq:chhomtpy_orbit}
\Phi_{1}=\Phi_{0} e^{(K\circ d_{1}-d_{0}\circ K)}, 
\end{equation}
where $d_{1}$ and $d_0$ are the differentials on $\mathcal{SC}^{+}(X_1,L_1)$ and $\mathcal{SC}^{+}(X_0,L_0)$, respectively.
\end{lemma}

\begin{rmk}
The exponential in \eqref{eq:chhomtpy_orbit} is the usual power series of operators.
\end{rmk}

\begin{rmk}
For the chord algebra $\mathcal{SC}^{+}(L)$,
Lemma~\ref{lem:chainhomotopy_orbit} follows from an extended version of~\cite[Lemma B.15]{Ekholm_rsft} (that takes orientations of the moduli spaces into account) which is stated in somewhat different terminology. In the proof below, we will adapt the terminology used there to the current setup so as to include (parameterized) orbits as well. Here, it should be mentioned that~\cite[Lemma B.15]{Ekholm_rsft}, and consequently also the current result, depend on a perturbation scheme for so-called M-polyfolds (the most basic level of polyfolds), the details of which are not yet worked out, and hence it should be viewed as a proof strategy rather than a proof in the strict sense.
\end{rmk}

We prove Lemma \ref{lem:chainhomotopy_orbit} in two steps. In the first step we relate $\Phi_0$ and $\Phi_1$ using an abstract perturbation that time orders the negative punctures in all moduli spaces of curves with punctures at chords and orbits in $C_{X_1}$. In the case that there are only chords there is a natural order of the negative punctures given by the boundary orientation of the disk and in that case the relation between $\Phi_0$ and $\Phi_1$ derived using the natural ordering perturbation can be turned into an algebraic relation. In the case that there are also orbits there is no natural ordering and to derive an algebraic formula we use all possible orderings and study homotopies relating different ordering perturbations. 

Consider the first step. We construct a perturbation that orders the negative boundary punctures of any curve in $\mathcal{F}^{I}_{1}$ (which is just a product with $\mathcal{F}_{1}^{s}\times I$ for any fixed $s\in I$) with negative punctures at chords in $C_{X_1}$. We need to carry out this perturbation energy level by energy level. Consider first the lowest action generator $\gamma$ of $H_1$ with action bigger than the chord at the positive puncture of the $(-2)$-disk. We perturb curves with positive puncture at $\gamma$ and with negative punctures at generators in $C_{X_1}$ by abstractly perturbing the Floer equation 
\[ 
(du-B_1)^{0,1}=0,
\]
near the negative punctures. Near chords and orbits in $C_{X_1}$ the data of the Floer equation is independent of both the $\R$-parameter and of $s\in I$. (Recall that the deformations are supported in the compact cobordism.) Thus, if the abstract time ordering perturbation is chosen sufficiently small then there are no $(d)$-curves for $d<0$ after perturbation and the moduli space of $(d)$-curves for $d\ge 0$ after abstract perturbation is canonically isomorphic to the corresponding moduli space before abstract perturbation. Assume that such a perturbation is fixed.

Let $\mathcal{G}(X_1,L_1)$ denote the set of generators of $\mathcal{SC}^{+}(X_1,L_1)$. For an element $\gamma\in \mathcal{G}(X_1,L_1)$ we write $d^{\varepsilon}_{1}\gamma$ for the sum of monomials that contribute to the differential of $\gamma$, i.e.~sum over $I$-components of the moduli spaces in $\mathcal{F}^{I}_{1}$, equipped with the additional structure of ordering of the generators as dictated by $\varepsilon$.

\begin{lemma}\label{lem:chainhomotopy_chord}
There is a map $K_{\varepsilon}\colon\mathcal{G}^{+}(X_1,L_1)\to \mathcal{SC}^{+}(X_0,L_0)$ such that for any generator $\gamma$ (chord or orbit),
\begin{equation}\label{eq:orderhomotopy}
\Phi_1(\gamma)-\Phi_0(\gamma)=\Omega_{K_\varepsilon}^{\varepsilon}(d_{1}^{\varepsilon}\gamma)+
d_0\Omega_{K_\varepsilon}^{\varepsilon}(\gamma).
\end{equation}
Here, $\Omega_{K_{\varepsilon}}^{\varepsilon}$ acts on monomials with an extra ordering of generators,  for an monomial of chords and orbits $\bo{\beta}=\beta_{1}\dots\beta_{k}$ we have
\begin{align*}
&\Omega_{K_\varepsilon}^{\varepsilon}(\bo{\beta})=\\ &\sum_{j=1}^{k}(-1)^{\tau_j}\Phi_{\sigma(1,j)}(\beta_1)\dots\Phi_{\sigma(1,j)}(\beta_{j-1})
K_{\varepsilon}(\beta_j)\Phi_{\sigma(j+1,j)}(\beta_{j+1})\dots\Phi_{\sigma(k,j)}(\beta_k),
\end{align*}
where $\tau_j=|\beta_1|+\dots+|\beta_{j-1}|$, $\sigma(i,j)=1$ if $\beta_i$ is before $\beta_j$ in the order perturbation $\varepsilon$ and $\sigma(i,j)=0$ if $\beta_i$ is after $\beta_j$.
\end{lemma}

\begin{proof}
Consider the parameterized moduli space
\[
\mathcal{F}_{\R}^{I}(\gamma;\bo{\beta})
\]
as above. Recall that $\mathcal{SC}^{+}(X_j,L_j)$ is defined as a direct limit using the action filtration corresponding to increasing slopes of Hamiltonians. We work below a fixed energy level with a fixed slope of our Hamiltonians and assume that the unique $(-2)$-disk forms a transversely cut-out $0$-manifold. 

We use the $(-2)$-disk to construct a chain homotopy. To this end we next extend the ordering perturbation $\varepsilon$ to all curves. Before we start the actual construction, we point out that our perturbation starts from the very degenerate situation where all negative boundary punctures lie at the same time. Thus one cannot avoid that new $(-2)$-disks arise when the perturbation is turned on. Gluing these to the perturbed moduli space of curves with negative asymptotes in $C_{X_1}$ then gives new $(-1)$-disks with positive puncture at $a$. We next show how to take these $(-1)$-curves into account.

We now turn to the description of the perturbation scheme. It is organized energy level by energy level in such a way that the size of the time separation of negative punctures of curves with positive and negative punctures in $C_{X_1}$ is determined by the action of the Reeb chord at the positive puncture. In particular the time distances between positive punctures of the newly created $(-2)$-disks at a given energy level are of the size of this time separation at this energy level. As we move to the next energy level, the time separation is a magnitude larger, so that one of the negative punctures of a disk on the new energy level passes all the positive punctures of the $(-2)$-disks created on lower energy levels before the puncture following it enters the region where $(-2)$-disks exists and can be attached to it.

Consider the parameterized $1$-dimensional moduli space $\mathcal{F}^{I}_{\R}(\gamma;\bo{\beta})$ of $(-1)$-disks defined using the perturbation scheme just described. The boundary of $\mathcal{F}^{I}_{\R}(\gamma;\bo{\beta})$ then consists of the $0$-manifolds $\mathcal{F}^{0}_{\R}(\gamma;\bo{\beta})$ and $\mathcal{F}^{1}_{\R}(\gamma;\bo{\beta})$ as well as broken disks that consist of one $(-2)$-disk at some $s\in I$ and several $(-1)$-disks with a $(0)$-disk in the upper or lower end attached. For a generator $\gamma$, let $K_{\epsilon}(\gamma)$ denote the count of $(-2)$-disks \emph{after} the ordering perturbation scheme described above is turned on:
\[
K_{\epsilon}(\gamma)=\sum_{|\gamma|-|\bo{\beta}|=-1}\frac{1}{m(\bo{\beta})!}\left|\mathcal{F}^{I}_{\R}(\gamma;\bo{\beta})\right|\bo{\beta},
\]
where $m(\bo{\beta})$ is the number of orbit generators in the monomial $\bo{\beta}$. To finish the proof we check that the $(-2)$ disks in the ordering perturbation scheme accurately accounts for the broken disks at the ends of the 1-dimensional moduli space. By construction, the separation of negative ends increases by a magnitude when we increase the energy level, and only one negative puncture of a curve in $\mathcal{F}^{I}_{1}$ can pass a $(-2)$-curve moment at a time. At the punctures which are ahead of this puncture with respect to $\varepsilon$, curves in $\Phi_1$ are attached, and at punctures which are behind it, curves in $\Phi_{0}$ are attached. Thus, counting the boundary points of oriented $1$-manifolds we conclude that~\eqref{eq:orderhomotopy} holds. 
\end{proof}

In order to express $\Phi_1$ in terms of $\Phi_0$ in an algebraically feasible way we discuss how the $(-2)$-disks counted by $K_\varepsilon$ depend on the choice of ordering perturbation $\varepsilon$. To this end we consider almost ordering perturbations $\varepsilon_u$, $u\in I$ which are time-ordering perturbations of the sort considered above of the negative ends of Floer curves in the positive end of the cobordism. Here an almost ordering is a true ordering except at isolated instances in $I$ when two ends are allowed to cross through with non-zero time derivative. Fix such a 1-parameter family and let $K_{\varepsilon_{0}}(\gamma)$ and $K_{\varepsilon_1}(\gamma)$ denote the count of $(-2)$-disks with positive puncture at $\gamma$ for the ordering perturbations $\varepsilon_0$ and $\varepsilon_1$. More precisely, we think of the whole 1-parameter family of moduli spaces associated to the orderings $\varepsilon_{u}$, $u\in I$. Consider now a path of paths corresponding to a disk $D$ interpolating between two orderings $\varepsilon_0$ and $\varepsilon_1$. More precisely, the boundary segment in the lower half plane in the boundary of the disk $D$ between $-1$ and $1$ corresponds to $\varepsilon_0$, and the boundary segment in the upper half plane corresponds to $\varepsilon_1$. 
The disk then interpolates between these two.

\begin{lemma}\label{l:nocross}
Generically there is a 1-dimensional locus $\Gamma$ in $D$ corresponding to $(-2)$-disks with transverse self-intersections and with boundary corresponding to $(-3)$-disk splittings, and at any $(-3)$-disk moment the path has a definite ordering (i.e.~no two negative ends are at the same time coordinate).
Furthermore, after deformation of $D$ we may assume that there are no self-intersections of $\Gamma$ (but that the disk still interpolates between the paths $\varepsilon_1$ and $\varepsilon_0$). 
\end{lemma}

\begin{proof}
The first part of the lemma is a straightforward transversality result. View the ordering paths as paths in larger dimensional spaces of problems where time coordinates are associated to the negative ends. Choosing these finite dimensional perturbations generically there is a transversely cut out $(-2)$-curve hypersurface in the larger spaces. The $(-2)$-disks in $D$ now correspond to intersections of the $(-2)$-disk hypersurfaces with $D$ considered as paths of paths in the larger spaces. For generic $D$ this then gives a curve $\Gamma$ with a natural compactification and with normal crossings. Endpoints of $\Gamma$ correspond to one $(-3)$-curve breaking off. Double points of $\Gamma$ correspond to two $(-2)$-curves which can be attached at the same disk with negative punctures in $C_{X_1}$. 

We next deform the disk $D$ in order to remove the double points of $\Gamma$. This is straightforward: closed components of $\Gamma$ bound disks in $D$ and can hence be shrunk by isotopy. Intersections of other types can be pushed across the boundary of $D$. This push results in two new intersections between the $(-2)$-curve hypersurface and a component of $\pa D$. These two intersections correspond to two copies of the same $(-2)$-curve with opposite signs and can be taken to lie arbitrarily close to each other. There is a third $(-2)$-curve between these two copies. However, by our original choice of abstract ordering perturbations all these three disks have positive puncture at almost the same moment in the $1$-parameter family in $\pa D$. For curves along  $\pa D$ with negative ends where these disks can be attached, the time-separation of these negative ends is then larger than the separation between the two $(-1)$-disks of opposite signs, and hence their contributions cancel.
\end{proof}

Consider the two counts of $(-2)$-disks $K_{\varepsilon}$ and $K_{\tau}$ corresponding to two ordering perturbations $\varepsilon$ and $\tau$. Lemma \ref{l:nocross} shows that there is a disk $D$ in which the $1$-manifold of $(-2)$-disks is embedded. Furthermore, if there are no $(-3)$-curves in $D$ the $1$-manifold of $(-2)$-disks gives a cobordism between the $(-2)$-disks along the boundary arcs and in this case $K_{\varepsilon}=K_{\tau}$. Thus, in order to relate in the general case, we only need to study what happens when the ordering path crosses a $(-3)$-disk moment.  Moreover, there is a fixed ordering of negative ends $\varepsilon'=\varepsilon$ or $\varepsilon'=\tau$ mapping to orbits in $C_{X_1}$ at such moments. Our next result expresses this change algebraically.

\begin{lma}\label{l:Kchange}
In the above setup, there is an operator $K_{\varepsilon\tau}$ such that
\begin{equation}\label{e:Kchange}
K_{\varepsilon}(\gamma)-K_{\tau}(\gamma)=
\Omega_{K_{\varepsilon\tau}}^{\varepsilon'}(d_{1}^{\varepsilon'}(\gamma))+d_0(\Omega_{K_{\varepsilon\tau}}^{\varepsilon'}(\gamma)).
\end{equation}
\end{lma}

\begin{proof}
The difference between $K_{\varepsilon}(\gamma)$ and $K_{\tau}(\gamma)$ corresponds to the intersection of $D$ and the codimension $2$ variety of $(-3)$-disks. The corresponding split disks are accounted for by the terms in the right hand side of \eqref{e:Kchange}.
\end{proof}

\begin{proof}[Proof of Lemma \ref{lem:chainhomotopy_orbit}]
By Lemma \ref{lem:chainhomotopy_chord} we have
\[
\Phi_{1}(\gamma)-\Phi_0(\gamma)=
\Omega_{K_{\varepsilon}}^{\varepsilon}(d_1^{\varepsilon}\gamma)+d_0\Omega_{K_{\varepsilon}}^{\varepsilon}(\gamma),
\]
where $\varepsilon$ corresponds to any ordering perturbation. We first show that we can replace $K_{\varepsilon}$ in this formula with $K_{\tau}$ for any ordering perturbation $\tau$. To this end we use Lemma~\ref{l:Kchange} which shows that with $\tau$ as there and $\varepsilon=\varepsilon'$ (otherwise exchange the roles of $\tau$ and $\varepsilon$), we have:
\begin{align*}
\Omega_{K_{\varepsilon}-K_{\tau}}^{\varepsilon}(d_1^{\varepsilon}\gamma)+d_0\Omega_{K_{\varepsilon}-K_{\tau}}^{\varepsilon}(\gamma)=&\; \Omega_{K_{\varepsilon\tau}d_1^{\varepsilon}}^{\varepsilon}(d_1^{\varepsilon}\gamma)+
\Omega_{d_0K_{\varepsilon\tau}}^{\varepsilon}(d_1^{\varepsilon}\gamma)\\
&+d_0d_0K_{\varepsilon\tau}(\gamma)
+d_0\Omega_{K_{\varepsilon\tau}}^{\varepsilon}(d_1^{\varepsilon}\gamma).
\end{align*}
Here the third term in the right hand side vanishes. We study the sum of the remaining three terms in the right hand side. 

The operator $\Omega_{K_{\varepsilon\tau}d_{1}^{\varepsilon}}^{\varepsilon}$ acts as follows on monomials $\beta_1\dots\beta_k$: act by $d_1^{\varepsilon}$ on $\beta_j$, attach $K_{\varepsilon\tau}$ at one of the arising negative punctures, attach $\Phi_{0}$ at all punctures before this puncture in $\varepsilon$ and  $\Phi_1$ at all punctures after. The sum
\[
d_0\Omega_{K_{\varepsilon\tau}}^{\varepsilon}(d_1^{\varepsilon}\gamma)+
\Omega_{d_0K_{\varepsilon\tau}}^{\varepsilon}(d_1^{\varepsilon}\gamma)
\]
counts configurations of the following form: act by $d_1^{\varepsilon}$ on $\gamma$, attach $K_{\varepsilon\tau}$ at one of its negative punctures, attach $\Phi_0$ or $\Phi_1$ at all remaining punctures, according to the ordering $\varepsilon$, then act by $d_0$ at the resulting negative punctures that do not come from $K_{\varepsilon\tau}$. (The terms in which $d_0$ acts on negative punctures of $K_{\varepsilon\tau}$ are counted twice with opposite signs in the above sum and hence cancel.) Using the chain map property of $\Phi_j$ we rewrite this instead as first acting with $d_1$ on the positive puncture where $\Phi_j$ was attached and then attaching $\Phi_j$ (and also remove the $d_0$ at the corresponding negative ends). We thus conclude that we can write the sum of the remaining terms in the right hand side as follows:
\begin{align*}
\Omega_{K_{\varepsilon\tau}d_1^{\varepsilon}}^{\varepsilon}(d_{1}^{\varepsilon}\gamma)+
\Omega_{d_0K_{\varepsilon\tau}}^{\varepsilon}(d_1^{\varepsilon}\gamma)
+d_0\Omega_{K_{\varepsilon\tau}}^{\varepsilon}(d_1^{\varepsilon}\gamma)=\Omega_{K_{\varepsilon\tau}}^{\varepsilon}(d_{1}^{\varepsilon}d_1^{\varepsilon}\gamma)=0,
\end{align*}
where the first term in the left hand side counts the terms where $K_{\varepsilon\tau}$ is attached at a negative end in the lower level curve in $d_1^{\varepsilon}d_1^{\varepsilon}$ and the sum of the last two counts the terms where it is attached at a negative end in the upper level. To see that $d_{1}^{\varepsilon}d_1^{\varepsilon}\gamma=0$ note that it counts the end points of an oriented compact 1-manifold.

We thus find that
\[
\Omega_{K_{\varepsilon}}^{\varepsilon}(d_1^{\varepsilon}\gamma)+ d_0\Omega_{K_{\varepsilon}}^{\varepsilon}(\gamma)=
\Omega_{K_{\tau}}^{\varepsilon}(d_{1}^{\varepsilon}\gamma)+d_0\Omega_{K_{\tau}}^{\varepsilon}(\gamma).
\]

Using this formula successively and noticing that if there are no $(-3)$-disks $K$ does not change over $D$, we find that 
\[
\Omega_{K_{\tau}}^{\varepsilon}(d_1^{\epsilon}\gamma)+d_0\Omega_{K_{\tau}}^{\varepsilon}(\gamma)=
\Omega_{K_{\tau}}^{\tau}(d_1^{\tau}\gamma)+d_0\Omega_{K_{\tau}}^{\tau}(\gamma).
\]

Thus for a specific ordering perturbation $\varepsilon$ we can move all the $\Phi_{0}$-factors across and using the splitting repeatedly we express the right hand side of \eqref{eq:chhomtpy_orbit} as the sum over all $r$-level trees, $r\ge 0$. Here $r$-level trees are defined inductively as follows. A $0$-level tree is a $\Phi_0$-curve. A $1$-level tree is a curve contributing to $d_1^{\varepsilon}$ with a $(-2)$-curve attached at one of its negative punctures and $\Phi_0$-curves at all others. An $r$-level tree is a curve contributing to $d_1^{\varepsilon}$ with a $(-2)$-curve attached at one of its negative punctures. At punctures after that, trees with $<r$ levels are attached. 

By the above we may take the $(-2)$-disks $K_{\varepsilon}=K$ to be independent of the ordering perturbation chosen and averaging over all ordering perturbations then gives
\[
\Phi_1(\gamma)=\Phi_0\,e^{(K d_1-d_0K)}(\gamma),
\]
by definition of the exponential.
\end{proof}

\begin{cor}
The chain maps induced by deformation equivalent cobordisms are chain homotopic.
\hfill{$\square$} 
\end{cor}

\subsection{Composition of cobordism maps}
We next study compositions of cobordism maps. Let $(X_{0},L_{0})$, $(X_1,L_1)$, and $(X_2,L_2)$ be Weinstein pairs, and let $(X_{01},L_{01})$ and $(X_{12},L_{12})$ be two cobordisms between $(X_{0},L_{0})$ and $(X_1,L_1)$ and between $(X_{1},L_{1})$ and $(X_2,L_2)$, respectively. We can then glue the cobordisms to form a cobordism $(X_{02},L_{02})$ from $(X_{0},L_{0})$ to $(X_{2},L_{2})$. This gives three cobordisms maps $\Phi_{01}$, $\Phi_{12}$, and $\Phi_{02}$ and we have the following result relating them:
\begin{theorem}
The chain maps $\Phi_{12}\circ\Phi_{01}$ and $\Phi_{02}$ are homotopic.
\end{theorem}   

\begin{proof}
Use the moduli space with two interpolation regions, and three Hamiltonians as shown in Figure~\ref{fig:H-2-cob}. When the $Y_{j}$, $j=0,1,2$ are maximally separated we find that the induced map equals the composition of the two maps. At the other end, when they collide, we have the composite map. The results in Section \ref{s:hmtpyofcobmaps} imply that the maps are homotopic.
\end{proof}

\begin{corollary}
The DGA $\mathcal{SC}^{+}(X,L)$ is invariant under exact isotopies.
\end{corollary}

\begin{proof}
Apply the homotopy of chain maps to the obvious deformation that takes the composition of the cobordism induced by an exact isotopy and its inverse to the trivial cobordism.
\end{proof}

\begin{figure}
         \begin{center}
\input{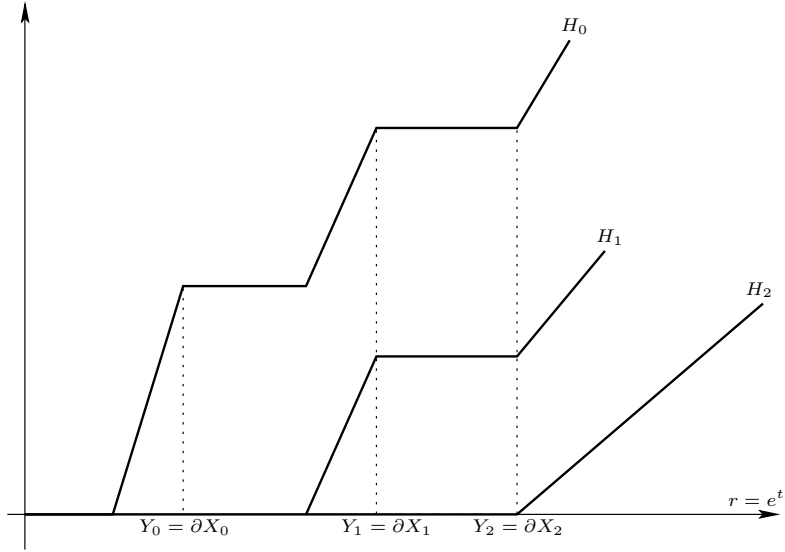}
         \end{center}
\caption{Hamiltonians for composition of cobordism maps.  \label{fig:H-2-cob}}
\end{figure}

\section{Isomorphism with contact homology} \label{sec:CHSH}
In this section we prove that the Hamiltonian simplex DGA $\mathcal{SC}^{+}(X,L)$ is quasi-isomorphic to the (non-equivariant) contact homology DGA $\mathcal{A}(Y,\Lambda)$ of its ideal boundary $(Y,\Lambda)$. 
The quasi-isomorphism is obtained using the cobordism map $\Phi$ in Section~\ref{s:fixHamcob} for vanishing Hamiltonian $H_1=0$. For the versions of contact homology where orbits and chords are not mixed this result implies that $\mathcal{SC}^{+}(L)$ is quasi-isomorphic to the Legendrian contact homology DGA of $\Lambda$, and that $\mathcal{SC}^{+}(X)$ is isomorphic to the (non-equivariant) contact homology DGA of $Y$. These results extend the corresponding isomorphisms between the high energy symplectic homology of $X$ and the non-equivariant linearized contact homology of $Y$~\cite{Bourgeois_Oancea_sequence}, or between the high energy wrapped Floer homology of $L$ and the linearized Legendrian homology of $\Lambda$, see e.g.~\cite{Ekholm_wrapped}, \cite[Theorem 7.2]{EHK}.

The non-equivariant orbit contact homology DGA is a natural generalization of the non-equivariant linearized contact homology, but is not described in the literature. We include a short description of the construction in Section~\ref{sec:nonequivDGA}. In Section~\ref{sec:equivDGA} we discuss the better known equivariant case that in our setup corresponds to the Hamiltonian simplex DGA associated to a time independent Hamiltonian. It should be mentioned that the transversality problems for the Floer equation corresponding to a time independent Hamiltonian are similar to the transversality problems for punctured holomorphic spheres in the symplectization end.    

\subsection{Non-equivariant contact homology orbit DGAs} \label{sec:nonequivDGA}
We give a brief description of non-equivariant contact homology. In essence this is simply a Morse-Bott theory for holomorphic disks and spheres with several negative interior punctures, where each Reeb orbit is viewed as a Morse-Bott manifold. (The chords are treated as usual, so our result for the Legendrian DGA is unaffected by the discussion here.)

Consider the contact manifold $Y$ which is the ideal contact boundary of $X$. To each Reeb orbit in $Y$ we will associate two decorated Reeb orbits $\hat{\gamma}$ and $\check{\gamma}$, see~\cite{Bourgeois_Oancea_sequence,BEE}. The gradings of these decorated orbits are  
\[  
|\check{\gamma}|=\CZ(\gamma)+(n-3)\quad \text{ and }\quad |\hat{\gamma}|=\CZ(\gamma)+(n-2).
\]
The differential in non-equivariant contact homology counts rigid Morse-Bott curves. These are several level holomorphic buildings where the asymptotic markers satisfy evaluation conditions with respect to a marked point on each Reeb orbit. Unlike in previous sections we here study curves in the symplectization. However, we still would like to use input from the filling. More precisely, as in~\cite{BEE,Bourgeois_Oancea_sequence} we will consider anchored curves. This means that all our curves have additional interior and boundary punctures where rigid holomorphic spheres and rigid holomorphic disks, respectively are attached. We will not mention the anchoring below but keep it implicit.

Recall first that in $D_{1;m,k}$ the positive boundary puncture determines an asymptotic marker at any interior negative puncture and that in $D_{0;0,k}$ any asymptotic marker at the positive puncture determines markers at all negative punctures. If $q$ is a puncture we write $\ev_{q}$ for the point on the Reeb orbit which is determined by the asymptotic marker. We next define Morse-Bott curves. 

Fix a point $x$ on each geometric Reeb orbit. A several level holomorphic curve with components $S_0,\dots, S_m$ with domain in $\mathcal{D}_{h;hm,k}$ is a \emph{Morse-Bott building} if the following holds:
\begin{itemize}
\item If the top level curve has a positive interior puncture $p$ then the following holds:
\begin{itemize}
\item If the asymptotic orbit is $\check{\gamma}$ then $\ev_{p}=x$.
\item If the asymptotic orbit is $\hat{\gamma}$ then $\ev_{p}$ is arbitrary.
\end{itemize}
\item 
For each component $S_j$ and for each negative interior puncture $q$ of $S_j$ the following holds:
\begin{itemize}
\item If there is a curve $S_m$ with positive interior puncture at $p$ attached to $S_j$ at $q$ then the oriented asymptotic Reeb orbit induces the cyclic order $(x,\ev_{q},\ev_{p})$ on the marked point and the two asymptotic markers.
\item If there is no curve attached at $q$ and the asymptotic orbit is $\hat{\gamma}$ then $\ev_{q}=x$.
\item If there is no curve attached at $q$ and the asymptotic orbit is $\check{\gamma}$ then $\ev_{q}$ is arbitrary.
\end{itemize}
\end{itemize}
  
Let $\mathcal{A}(Y,\Lambda)$ denote the graded unital algebra generated by the Reeb chords of $\Lambda$ and decorated Reeb orbits, where as above chords and decorated orbits sign commute and where decorated orbits sign commute with each other. The differential on $\mathcal{A}(Y,\Lambda)$ is given by a holomorphic curve count. Using notation analogous to the above we write $\mathcal{M}(a;\mathbf{b},\tilde{\bo{\eta}})$, $\mathbf{b}=b_1\dots b_m$, $\tilde{\bo{\eta}}=\tilde\eta_1\dots\tilde\eta_k$, where $\tilde\eta_j$ is a decorated orbit, for the moduli space of anchored Morse-Bott curves $u\colon D_{1;m,k}\to (\R\times Y,\R\times\Lambda)$ with positive boundary puncture where the map is asymptotic at $\infty$ to the holomorphic Reeb chord strip $\R\times a$, and $m$ negative boundary punctures and $k$ negative interior punctures where the map is asymptotic to  the Reeb chord strips $\R\times b_j$ and the Reeb orbit cylinders $\R\times\eta_j$ at $-\infty$. Similarly, we write $\mathcal{M}(\tilde{\gamma},\tilde{\bo{\eta}})$ for the moduli space of anchored Morse-Bott curves $u\colon D_{0;0,k}\to \R\times Y$ with positive interior puncture where the map is asymptotic at $\infty$ to the holomorphic Reeb orbit cylinder $\R\times \gamma$, and $k$ negative interior punctures where the map is asymptotic to the Reeb orbit cylinders $\R\times\eta_j$ at $-\infty$.

Define the differential $d$ to satisfy the Leibniz rule and to be given as follows on generators: for chords
\[ 
d a =\sum_{|a|-|\mathbf{b}|-|\tilde{\bo{\eta}}|=1}\frac{1}{k!}|\mathcal{M}(a;\mathbf{b},\tilde{\bo{\eta}})|\mathbf{b}\tilde{\bo{\eta}},
\]    
and for orbits
\[
d \tilde\gamma =\sum_{|\tilde\gamma|-|\tilde{\bo{\eta}}|=1}\frac{1}{k!}|\mathcal{M}(\tilde\gamma;\tilde{\bo{\eta}})|\tilde{\bo{\eta}}.
\]
Here $|\mathcal{M}|$ denotes a sign count of elements of a rigid moduli space with respect to a system of coherent orientations and $m(\bo{\eta})$ is the number of orbits in the monomial $\bo{\eta}$. See~\cite[\S4.4]{BOauto} for a discussion of orientations for fibered products that is relevant in the case at hand. Then, much like in Lemma~\ref{lma:secondarydelta2}, we have $d^{2}=0$.

\begin{remark}
Instead of using the Morse-Bott framework above, one can give an alternative definition of the non-equivariant DGA $\mathcal{A}(Y,\Lambda)$ by considering gluing compatible almost complex structures which are time-dependent and periodic in cylindrical end coordinates near interior punctures, i.e.~$J=J_t$, $t\in S^1$. The relevant moduli spaces would then have to be defined in terms of asymptotic incidence conditions determined by a choice of reference point on each periodic Reeb orbit. 
\end{remark}

We now set $H_1=0$ in the construction presented in Section~\ref{s:fixHamcob} and replace the cobordism there by the trivial cobordism, i.e.~the symplectization of $Y$. Define the algebra map $\Psi\colon \mathcal{A}(Y,\Lambda)\to\mathcal{SC}^+(X,L;H_0,J_0)$ as follows on generators: for chords
\[ 
\Psi(a)=\sum_{|a|-|\mathbf{b}|-|{\bo{\eta}}|=0} \frac 1 {k!} |\mathcal{F}_{\R}(a;\mathbf{b},\bo{\eta})|\,\bo{\eta}\mathbf{b},
\] 
and for orbits
\[
\Psi(\tilde\gamma)=\sum_{|\tilde\gamma|-|\bo{\eta}|=0} \frac 1 {k!} |\mathcal{F}_{\R}(\tilde{\gamma};\bo{\eta})|\,\bo{\eta}.
\]
Passing to the direct limit as the slope of $H_0$ goes to infinity we obtain an induced map 
$$
\Psi:\mathcal{A}(Y,\Lambda)\to\mathcal{SC}^+(X,L).
$$

\begin{theorem}\label{t:iso}
The induced map $\Psi\colon \mathcal{A}(Y,\Lambda)\to\mathcal{SC}^+(X,L)$ is a chain isomorphism.
\end{theorem}

\begin{proof}
The fact that $\Psi$ is a chain map follows as usual by identifying contributing terms of $d\Psi-\Psi d$ with the endpoints of a $1$-dimensional moduli space. The isomorphism statement is a consequence of the fact that the Reeb chord strips or Reeb orbit cylinders contribute $1$, together with a standard action-filtration argument. For orbits, see ~\cite{Bourgeois_Oancea_sequence}. For chords the strips of the isomorphism are simply reparameterizations of trivial strips over the Reeb chord cut off in the $r$-slice where the corresponding Hamiltonian chord lies. 
\end{proof}

\subsection{Equivariance and autonomous Hamiltonians} \label{sec:equivDGA}
In order to relate the usual (equivariant) contact homology $\widetilde{\mathcal{A}}(Y,\Lambda)$ of the ideal contact boundary to a Hamiltonian simplex DGA we can use more or less the same argument as in the non-equivariant case. The starting point here is to set up an equivariant version of the Hamiltonian simplex DGA. To this end we use a time-independent one step Hamiltonian and define a version $\widetilde{\mathcal{SC}}^{+}(X,L)$ of the secondary symplectic DGA generated by unparameterized Hamiltonian orbits. To establish transversality for this theory one needs to use abstract perturbations. Assuming that such a perturbation scheme -- that also extends to curves in the symplectization with no Hamiltonian -- has been fixed, we can repeat the constructions of Section~\ref{sec:nonequivDGA} word by word to prove   

\begin{theorem}\label{t:isoequiv}
The map $\widetilde{\Psi}\colon \widetilde{\mathcal{A}}(Y,\Lambda)\to\widetilde{\mathcal{SC}}^+(X,L)$ is a chain isomorphism.
\end{theorem}

\begin{proof}
Analogous to Theorem \ref{t:iso}.
\end{proof}

\section{Examples and further developments} \label{sec:examples}
In this section we first discuss examples where the Hamiltonian simplex DGA is known via the isomorphism to contact homology. Then we discuss how the theory can be generalized to connect Hamiltonian Floer theory to other parts of SFT.

\subsection{Knot contact homology.}
Our first class of examples comes from Legendrian contact homology. By Theorem~\ref{t:iso}, this corresponds to the chord case of our secondary DGA. 

Given a knot $K\subset S^3$, one considers its conormal bundle $\nu K\subset X=T^*S^3$. This is an exact Lagrangian that is conical at infinity, that has vanishing Maslov class, and whose wrapped Floer homology $WH(\nu K)$ was shown in~\cite{APS} to be equal to the homology of the space $\mathcal{P}_K S^3$ of paths in $S^3$ with endpoints on $K$, i.e. 
$$
WH(\nu K)\simeq H(\mathcal{P}_K S^3). 
$$
One can prove that the homotopy type of the space $\mathcal P_K S^3$ does not change as the knot is deformed in a $1$-parameter family possibly containing immersions. Since any two knots in $S^3$ can be connected by a path that consists of embeddings except at a finite number of values of the deformation parameter, where it consists of immersions with a single double point, we infer that $\mathcal P_K S^3$ has the same homotopy type as $\mathcal P_U S^3$, where $U\subset S^3$ is the unknot. As a matter of fact, the homotopy equivalence can be chosen to be compatible with the evaluation maps at the endpoints, showing that $H(\mathcal P_K S^3)\simeq H(\mathcal P_U S^3)$ as algebras with respect to the Pontryagin-Chas-Sullivan product. Moreover, we also have isomorphisms $H(\mathcal P_K S^3,K)\simeq H(\mathcal P_U S^3,U)$ induced by homotopy equivalences. We infer that $WH(\nu K)$ and its high-energy version $\mathcal{WH}^+(\nu K)\simeq H(\mathcal P_K S^3,K)$ are too weak as invariants in order to distinguish knots.  

In contrast, for the superficially different case $K\subset \mathbb{R}^3$, the Legendrian contact homology of $\nu K$, also called \emph{knot contact homology}, was proved in~\cite{EENS} to coincide with the combinatorial version of~\cite{Ngframed} and, as such, to detect the unknot. Theorem~\ref{t:iso} can be extended in a straightforward way to cover the case of $T^*\mathbb{R}^3$ in order to show that Legendrian contact homology of $\nu K$ is isomorphic to the homology of the Hamiltonian simplex DGA $\mathcal{SC}^+(\nu K)$. In particular, the operations on wrapped Floer homology corresponding to the differential on $\mathcal{SC}^+(\nu K)$ are rich and interesting. This contrasts to the naive higher co-products which are rather trivial. In terms of $\mathcal{P}_{K}$ the operations of the Hamiltonain simplex DGA correspond to fixing points on the paths with endpoints on $K$ and then averaging over the locations of these points. This gives a string topological interpretation of knot contact homology, where chains of strings split as the strings cross the knot as studied in~\cite{CELN}. 

As a final remark, the coefficient ring of knot contact homology involves a relative second homology group that in the unit cotangent bundle of $\R^{3}$ contains also the class of the fiber, which is killed in the full cotangent bundle. This extra variable is key to the relation between knot contact homology and the topological string, see \cite{AENV}, and indicates that it would be important to study the extension of the theory described in the current paper to a situation where the contact data at infinity does not have any symplectic fillings.

\subsection{$A_\infty$, $L_\infty$, and the diagonal}
As already mentioned in the introduction, the Hamiltonian simplex DGA $\mathcal{SC}^+(X)$ in the orbit case can be viewed as the cobar construction on the vector space generated by the high-energy orbits, viewed as an $\infty$-Lie coalgebra with the sequence of operations $(d_1,d_2,\dots)$. Note that $\infty$-Lie coalgebras are dual to $L_\infty$, or $\infty$-Lie algebras. 

In a similar vein, given a Lagrangian $L\subset X$ the Hamiltonian simplex DGA $\mathcal{SC}^+(L)$ in the chord case can be viewed as the cobar construction on the vector space generated by the high-energy chords, viewed as an $\infty$-coalgebra, a type of structure that is dual to $A_\infty$-algebras. 

It turns out that one can produce an $\infty$-algebra structure in the orbit case by implementing exactly the same construction subject to the additional condition that all punctures lie on a circle on the sphere. This condition is invariant under conformal transformations and yields well-defined moduli spaces, which effectively appear as submanifolds inside the moduli space that define the operations on $\mathcal{SC}^+(X)$. The resulting DGA is not an $\infty$-Lie coalgebra, but simply an $\infty$-coalgebra. 

Doubling chords to orbits and holomorphic disks to holomorphic spheres with punctures on a circle using Schwarz reflection, it is straightforward to show that that the resulting DGA coincides with $\mathcal{SC}^+(\Delta_X)$, the Hamiltonian simplex DGA of the Lagrangian diagonal $\Delta_X\subset X\times X$. This fact parallels the well-known isomorphism between periodic Hamiltonian Floer homology and Lagrangian Floer homology of the diagonal.

This example shows in particular that the relationship between the Hamiltonian simplex DGAs in the closed and in the open case is subtler than its linear counterpart. 

\subsection{Chern class, Maslov class, and exactness}
We discuss in this section some of the standing assumptions in the paper. 

A first set of assumptions imposed in Section~\ref{S:opendga} is that $\pi_1(X)=0$, $c_1(X)=0$, $\mu_L=0$. These are the simplest technical assumptions under which the theory is $\mathbb{Z}$-graded. If $\pi_1(X)=0$ but $c_1$ or $\mu_L$ are non-zero, the closed theory would be graded modulo the positive generator of $c_1(X)\cdot H_2(X)$, and the open theory would be graded modulo the positive generator of $\mu_L\cdot H_2(X,L)$. There are also ways to dispose of the condition $\pi_1(X)=0$ at the expense of possibly further weakening the grading, see the discussion in~\cite{EGH}. 

A standing assumption of a quite different and much more fundamental kind is that the manifold $X$ and the Lagrangian $L$ be exact. This is a simple way to rule out \emph{a priori}\,Êthe bubbling-off of pseudo-holomorphic spheres in $X$, respectively of pseudo-holomorphic discs with boundary on $L$. The advantage of this simple setup is that it allows us to focus on the new algebraic structure. The theory would need to be significantly adapted should one like to consider non-exact situations.

\subsection{Further developments}
At a linear level, $S^1$-equivariant symplectic homology is obtained from its non-equivariant counterpart using (an $\infty$-version of) the $BV$-operator~\cite{BOcontact-equivariant}. The $BV$-operator is an operation governed by the fundamental class of the moduli space of spheres with two punctures and varying asymptotic markers at the punctures. Note that this moduli space is homeomorphic to a circle and the $BV$-operator has degree $+1$ as a homological operation, which corresponds to the fact that the fundamental class of the moduli space lives in degree $1$. It was proved in~\cite{BOcontact-equivariant} that the high-energy, or positive part of $S^1$-equivariant symplectic homology recovers linearized cylindrical contact homology of the contact boundary $Y$. One advantage of the $S^1$-equivariant point of view over the symplectic field theory (SFT) point of view is that it does not require any abstract perturbation theory. 

\begin{question}
What is the additional structure on the non-equivariant Hamiltonian simplex DGA $\mathcal{SC}^+(X)$ that allows to recover the equivariant Hamiltonian simplex DGA described in Section~\ref{sec:equivDGA}?
\end{question}  

Though one can construct an $\infty$-version of the BV-operator in the DGA setting that we consider in this paper by methods similar to those of~\cite{BOcontact-equivariant}, it is not clear whether this is enough in order to recover the equivariant DGA from the non-equivariant one. It may be that one needs more information coming from the structure of an algebra over the operad of framed little $2$-disks that exists on any Hamiltonian Floer theory. 

From the point of view of SFT, the natural next steps are to understand the algebraic structure that is determined on Hamiltonian Floer theory by moduli spaces of genus $0$ curves with an arbitrary number of positive punctures, respectively by moduli spaces of curves with an arbitrary number of positive punctures and arbitrary genus. These would provide in particular non-equivariant analogues of the rational SFT and full SFT.

\appendix
\section{Determinant bundles and signs}\label{A:1}
In this appendix we give a more detailed discussion of how the sign rules of the Hamiltonian simplex DGA derive from orientations of determinant bundles. The material here has been discussed at many places in this context, see for example~\cite[\S11]{Seidel-book}, \cite[Appendix~A.2]{MS04}, \cite{Bourgeois-Mohnke}, \cite{Zinger-det}, \cite{FO3}, \cite{EES_ori}. 

If $V$ is a finite dimensional vector space then $\Lambda^{\max}V=\Lambda^{\dim\, V}V$ is its highest exterior power. For the $0$-dimensional vector space, $\Lambda^{\max}(0)=\R$. If
\[
\begin{CD}
0 @>>> V_1 @>{f_1}>> V_2 @>{f_2}>> \dots @>{f_n}>> V_{n+1} @>>> 0
\end{CD}
\]
is an exact sequence of finite dimensional vector spaces then there is a canonical isomorphism
\[
\bigotimes_{k \text{ odd}}\Lambda^{\max}V_k \  \ \cong \ \ \bigotimes_{k \text{ even}}\Lambda^{\max}V_k
\]
that \emph{does} depend on the maps $f_1,\dots,f_n$. For example, if $\dim\, V_1$ is odd and the map $f_1$ is changed to $-f_1$, then the isomorphism changes sign. 

If $X$ and $Y$ are Banach spaces and $D\colon X\to Y$ is a Fredholm operator, then the \emph{determinant line $\det(D)$ of $D$} is the 1-dimensional vector space 
\[
\det(D)=\Lambda^{\max}(\coker D)^*\otimes \Lambda^{\max}\ker D.
\]
We think of $\det(D)$ as a graded vector space supported in degree $\ind(D)$. 

We next discuss stabilization. We first stabilize in the source. Let $D\colon X\to Y$ be a Fredholm operator, $V$ a finite dimensional real vector space and $\Phi\colon V\to Y$ a linear map. The \emph{stabilization of $D$ by $\Phi$} is the Fredholm operator $D^V=D\oplus\Phi\colon X\times V\to Y$, $(x,v)\mapsto Dx+\Phi v$. The exact sequence 
$$
0\longrightarrow \ker D \longrightarrow \ker D^V\longrightarrow V\stackrel{\Phi}\longrightarrow \coker D\longrightarrow \coker D^V\to 0
$$
gives a canonical isomorphism  
$$
\det(D^V)\cong \det D\otimes \Lambda^{\max}V, 
$$
that depends on the map $\Phi$. For example, if the map $\Phi$ changes sign and if $\dim\coker D-\dim\coker D^V$ is odd, then the isomorphism changes sign. 

Similarly we can stabilize in the target. If $W$ is a finite dimensional vector space and $\Psi\colon X\to W$ is a continuous linear map, then with $D_{W}=(D,\Psi)\colon X\to Y\times W$, $x\mapsto (Dx,\Psi x)$ we get 
$$
0\longrightarrow \ker D_W \longrightarrow \ker D\stackrel{\Psi}\longrightarrow W\longrightarrow \coker D_W\longrightarrow \coker D\to 0,
$$
which gives a canonical isomorphism 
$$
\det(D_W)\cong  (\Lambda^{\max} W)^{\ast}\otimes\det D. 
$$
that depends on $\Psi$. For example, if the map $\Psi$ changes sign and $\dim\ker D - \dim\ker D_W$ is odd, then the isomorphism changes sign. 

Finally, combining the two if $\alpha\colon V\to W$ is a linear map then for the map $D^{V}_{W}\colon X\times V\to Y\times W$, $D^{V}_{W}(x,v)=(Dx+\Phi v, \Psi x + \alpha v)$ gives a canonical isomorphism
\begin{equation}\label{eq:detiso}
\det(D_W^{V})\cong (\Lambda^{\max} W)^{\ast}\otimes\det D\otimes \Lambda^{\max} V 
\end{equation}
that depends on $\Phi$, $\Psi$, and $\alpha$.

\begin{remark}
For the isomorphism above one also needs to specify conventions for orientations of direct sums corresponding to stabilizations. The details of these conventions do however not affect our discussion here.
\end{remark}

If $D\in \mathcal{F}(X,Y)$ then, by stabilizing in the domain, one may make all operators in a neighborhood of $D$ surjective and that together with the above isomorphism allows for the definition of a locally trivial line bundle $\underline{\det}\to\mathcal{F}(X,Y)$ over the space of Fredholm operators acting from $X$ to $Y$ with fiber over $D$ equal to $\det(D)$.

Assume $\mathcal{D}\colon \mathcal{O}\to \mathcal{F}(X,Y)$ is a continuous map defined on some topological space $\mathcal{O}$. Consider the pull-back bundle $\mathcal{D}^*\underline{\det}\to \mathcal{O}$ and note that it admits a trivialization provided the first Stiefel-Whitney class vanishes, $w_1(\mathcal{D}^*\underline{\det})=0$.

If $V$ and $W$ are finite dimensional vector spaces, we consider in line with the discussion above the bundle $\mathcal{O}^{V}_{W}=\mathcal{O}\times \Hom(V,Y)\times \Hom(X,W)\times \Hom(V,W)$ and the map $\mathcal{D}^{V}_{W}\colon \mathcal{O}^{V}_{W}\to\mathcal{F}(X\times V,Y\times W)$ defined as follows: $\mathcal{D}^{V}_{W}(p,\Phi,\Psi,\alpha)$ is the linear map which takes $(x,v)\in X\times V$ to
\[ 
(\mathcal{D}(p)x+\Phi v,\Psi x+\alpha v)\in X\times W.
\]
Using the natural base point in $\Hom(V,Y)\times \Hom(X,W)\times \Hom(V,W)$ given by $(\Phi,\Psi,\alpha)=0$ and the natural isomorphism \eqref{eq:detiso}, we transport orientations of $\Lambda^{\max}\underline{W}^*\otimes \mathcal{D}^*\underline{\det}\otimes\Lambda^{\max}\underline{V}\to\mathcal{O}$ to orientations of $(\mathcal{D}^{V}_{W})^*\underline{\det}\to\mathcal{O}^{V}_{W}$, and back. Here $\underline V\to \mathcal{O}$ and $\underline{W}\to \mathcal{O}$ denote the trivial bundles $\mathcal{O}\times V\to\mathcal{O}$ and $\mathcal{O}\times W\to\mathcal{O}$ respectively.

We now apply this setup to spaces of (stabilized) Cauchy-Riemann operators used in the definition of the Hamiltonian simplex DGA. Indeed, the linearized operator for our family of Cauchy-Riemann equations parameterized by the simplex $\Delta^{m-1}$ is of the type $D^{T_w\Delta^{m-1}}$. Since $T_w\Delta^{m-1}=\ker \ell_{\nu}$ with $\ell_{\nu}\colon \R^m\to \R$, $\ell_{\nu}(\zeta)=\langle \zeta,\nu\rangle$, where $\nu$ is the vector $\nu=(1,1,\dots, 1)$, we have a canonical isomorphism $\det(D^{T_w\Delta^{m-1}})\simeq \det(D^{\R^m}_{\R})$, with $D^{\R^m}_{\R}(x,\zeta)=(D^{T_w\Delta^{m-1}}(x,\pi\zeta),\ell_{\nu}(\zeta))$ and $\pi\colon \R^m\to T_w\Delta^{m-1}$ the orthogonal projection parallel to $\nu$. We can thus view the linearization of our parameterized Cauchy-Riemann problem as an element of a suitable space $\mathcal{O}^{\R^m}_{\R}$ of Fredholm operators of Cauchy-Riemann type. 

The negative orbit and chord capping operators $o_-(\gamma)$, $o_-(c)$ belong to natural spaces $\mathcal{O}_-(\gamma)$, $\mathcal{O}_-(c)$ of Cauchy-Riemann operators with fixed asymptotic behavior determined by the linearized Hamiltonian flow along $\gamma$ and respectively $c$, acting between appropriate Sobolev spaces of sections $W^{1,p}\to L^p$, $p>2$ (see~\cite[\S4.4]{BOauto} for the orbit case and~\cite[\S11]{Seidel-book}, \cite[\S9]{Abouzaid-Seidel} for the chord case). These spaces of Cauchy-Riemann operators with fixed asymptotes are contractible, and consequently the determinant line bundle can be trivialized over each of them. We similarly define natural spaces of Cauchy-Riemann operators $\mathcal{O}_+(\gamma)$, $\mathcal{O}_-(c)$ containing the positive orbit and chord capping operators $o_+(\gamma)$, $o_+(c)$. 

Our procedure for the construction of coherent orientations for the parameterized Cauchy-Riemann equation is then the following. 

\begin{itemize}
\item[(i)]
Given the canonical orientation on $\C$, we orient the determinant bundles over the spaces $\mathcal{O}(\C P^1)$ of Cauchy-Riemann operators over $\C P^1$ by the canonical orientation of complex linear operators. Since all the Euclidean spaces $\R^n$ are canonically oriented, this induces orientations of the determinant bundles over all spaces $\mathcal{O}^{\R^k}_{\R^\ell}(\C P^1)$ for arbitrary $k,\ell\in\Z_{\ge 0}$. 

Similarly, following \cite{FO3}, the choice of a relative spin structure on the Lagrangian $L$ determines an orientation of the determinant bundle over all spaces of Cauchy-Riemann operators $\mathcal{O}(D^2)$ defined on the pull-back of $TX$ over the disk $D^2$ by arbitrary smooth maps $u\colon(D^2,\p D^2)\to (X,L)$, with totally real boundary conditions given by $u|_{\p D^2}^*TL$. 
This then induces orientations of the determinant bundles over all spaces $\mathcal{O}^{\R^k}_{\R^\ell}(D^2)$ for arbitrary $k,\ell\in\Z_{\ge 0}$. 

\item[(ii)] We choose orientations of the determinant lines over the spaces $\mathcal{O}_-(\gamma)$, $\mathcal{O}_-(c)$, which determine in turn orientations of the determinant lines over all spaces ${\mathcal{O}_-}^{\R}_{0}(\gamma)$, ${\mathcal{O}_-}^{\R}_{0}(c)$. 

\item[(iii)] By gluing, we obtain orientations of the determinant lines over the spaces $\mathcal{O}_+(\gamma)$, $\mathcal{O}_+(c)$, which determine in turn orientations of the determinant lines over all spaces ${\mathcal{O}_+}^{\R}_{0}(\gamma)$, ${\mathcal{O}_+}^{\R}_{0}(c)$. 

\item[(iv)] If $\mathbf{b}=b_1\dots b_m$ is a word of Hamiltonian chords and $\bo{\eta}=\eta_1\dots\eta_k$ a word of Hamiltonian orbits, then we write
\[ 
{\mathcal{O}_+}^{\R}_{0}(\mathbf{b},\bo{\eta})={\mathcal{O}_+}^{\R}_{0}(b_1)\times\dots\times {\mathcal{O}_+}^{\R}_{0}(b_m)\times {\mathcal{O}_+}^{\R}_{0}(\eta_1)\times\dots\times {\mathcal{O}_+}^{\R}_{0}(\eta_k)
\]
and 
\[  
{\mathcal{O}_+}^{\R}_{0}(\bo{\eta})={\mathcal{O}_+}^{\R}_{0}(\eta_1)\times\dots\times {\mathcal{O}_+}^{\R}_{0}(\eta_k).
\] 
Given a Hamiltonian chord $a$, we write $\mathcal{O}(a;\mathbf{b},\bo{\eta})$ for the space of Cauchy-Riemann operators defined on a punctured disc with one positive boundary puncture, $m$ negative boundary punctures, and $k$ negative interior punctures, with Lagrangian boundary conditions given by the pull-back of $TL$ via a map on the disk into $X$ with boundary in $L$, and with asymptotic behavior at the punctures according to the Hamiltonian chords and orbits $a$, $\mathbf{b}$, and $\bo{\eta}$. Similarly, given a Hamiltonian orbit $\gamma$ we write $\mathcal{O}(\gamma;\bo{\eta})$ for the space of Cauchy-Riemann operators defined on a sphere with one positive puncture and $k$ negative punctures, and with asymptotic behavior at the punctures determined by the linearized flow along the Hamiltonian orbits $\gamma$, $\bo{\eta}$. We then have spaces $\mathcal{O}^{\R^{m+k+1}}_{0}(a;\mathbf{b},\bo{\eta})$ and $\mathcal{O}^{\R^{k+1}}_{0}(\gamma;\bo{\eta})$, respectively $\mathcal{O}^{\R^{m+k}}_{\R}(a;\mathbf{b},\bo{\eta})$ and $\mathcal{O}^{\R^k}_{\R}(\gamma;\bo{\eta})$. 
\item[(v)] Cauchy-Riemann operators which are stabilized by finite dimensional spaces at the source 
can be glued much like usual, i.e.~non stabilized, Cauchy-Riemann operators, see e.g.~\cite[Section 4.3]{ESm}. The gluing operations 
$$
{\mathcal{O}_-}^{\R}_{0}(a)\times\mathcal{O}^{\R^{m+k+1}}_{0}(a;\mathbf{b},\bo{\eta})\times {\mathcal{O}_+}^{\R}_{0}(\mathbf{b},\bo{\eta})\longrightarrow \mathcal{O}^{\R^{2(m+k+1)}}_{0}(D^2)
$$
and 
$$
{\mathcal{O}_-}^{\R}_{0}(\gamma)\times \mathcal{O}^{\R^{k+1}}_{0}(\gamma;\bo{\eta})\times {\mathcal{O}_+}^{\R}_{0}(\bo{\eta})\longrightarrow\longrightarrow \mathcal{O}^{\R^{2(k+1)}}_{0}(\C P^1)
$$
induce isomorphisms of determinant bundles which are canonical up to homotopy. From our previous choices we obtain orientations of all the spaces $\mathcal{O}^{\R^{m+k+1}}_{0}(a;\mathbf{b},\bo{\eta})$ and $\mathcal{O}^{\R^{k+1}}_{0}(\gamma;\bo{\eta})$. After restricting to the slice given by the zero stabilization map, we obtain as explained above orientations of all the spaces $\mathcal{O}^{\R^{m+k}}_{\R}(a;\mathbf{b},\bo{\eta})$ and $\mathcal{O}^{\R^{k}}_{\R}(\gamma;\bo{\eta})$. These orientations are used in order to count  rigid holomorphic curves with signs in the relevant moduli spaces. 
\end{itemize}

Our choice of coherent orientations gives the following graded commutativity property. As in Section~\ref{sec:Floermodulispaces}, let $|c|=\CZ(c)-2$ and $|\gamma|=\CZ(\gamma)+n-3$. Let $\mathbf{b}=b_1\dots b_m$, $\bo{\eta}=\eta_1\dots\eta_k$ be words in Hamiltonian chords and orbits respectively as above. Consider spaces of stabilized Cauchy-Riemann operators $\mathcal{O}^{\R^{m+k+1}}_{0}(a;\mathbf{b},\bo{\eta})$ and $\mathcal{O}^{\R^{k+1}}_{0}(\gamma;\bo{\eta})$ for Hamiltonian chords and orbits $a$ and $\gamma$. Given $1\le i\le k-1$ denote $\bo{\eta}^i=\eta_1\dots\eta_{i-1}\eta_{i+1}\eta_i\eta_{i+2}\dots\eta_k$. There are canonical identifications
\[ \mathcal{O}^{\R^{m+k+1}}_{0}(a;\mathbf{b},\bo{\eta})\cong\mathcal{O}^{\R^{m+k+1}}_{0}(a;\mathbf{b},\bo{\eta}^i) 
\]  
and 
\[
\mathcal{O}^{\R^{k+1}}_{0}(\gamma;\bo{\eta})\cong \mathcal{O}^{\R^{k+1}}_{0}(\gamma;\bo{\eta}^i) 
\]
obtained by relabeling the $i$-th and $(i+1)$-th interior punctures of the domain. Accordingly, the determinant line bundles over these spaces of operators are canonically identified. Each of them comes with an induced orientation as above, and these orientations differ by the sign 
$$
(-1)^{|\eta_i||\eta_{i+1}|}.
$$
Indeed, these orientations differ by the same sign as the orientations of the determinant lines over ${\mathcal{O}_+}^{\R}_{0}(\gamma_i)\times\mathcal{O}^{\R}_{0}(\gamma_{i+1})$ and ${\mathcal{O}_+}^{\R}_{0}(\gamma_{i+1})\times\mathcal{O}^{\R}_{0}(\gamma_i)$, identified via the obvious exchange of factors. The latter sign is equal to~\cite[p.150]{Seidel-book} 
$$
(-1)^{\ind({D^{\R}_{0}}_i)\times\ind({D^{\R}_{0}}_{i+1})}=(-1)^{|\eta_i||\eta_{i+1}|},
$$
where ${D^{\R}_{0}}_i\in {\mathcal{O}_+}^{\R}_{0}(\gamma_i)$ and ${D^{\R}_{0}}_{i+1}\in {\mathcal{O}_+}^{\R}_{0}(\gamma_{i+1})$. This holds because 
$$
\ind({D^{\R}_{0}}_i)=\CZ(\eta_i)+n+1\equiv |\eta_i| (\mathrm{mod}\, 2)
$$ 
and 
$$
\ind({D^{\R}_{0}}_{i+1})=\CZ(\eta_{i+1})+n+1\equiv |\eta_{i+1}| (\mathrm{mod}\, 2),
$$ 
see Section~\ref{sec:Floermodulispaces}.

This shows that orbits sign commute in the Hamiltonian simplex DGA of Section~\ref{S:opendga}.

\begin{remark} \label{rmk:sign_commute} 
In the Hamiltonian simplex DGA of Section~\ref{S:opendga} orbits sign commute with chords. That is \emph{not} a consequence of coherent orientations. It is just an algebraic choice that reflects the interpretation of orbits as coefficients for the algebra generated by chord generators. Indeed, we can always order the negative punctures of a holomorphic curve by first considering boundary punctures and then considering interior punctures (analogous to normal ordering of operators). 
\end{remark}

\bibliographystyle{abbrv}
\bibliography{000_refs_sh_ch}	

\end{document}